\documentclass[a4paper,7pt,oneside,onecolumn,number,preprint,centertitle]{elsarticle}

\usepackage{amsmath,amssymb,bm}
\usepackage{graphicx}
\usepackage{textcomp}
\usepackage{xcolor}
\usepackage{amsthm}
\usepackage{tikz}
\usetikzlibrary{positioning}
\usetikzlibrary{arrows.meta}

\usepackage{enumerate}
\usepackage{bbold}
\usepackage{epsfig}
\usepackage{multicol}
\usepackage{float}
\usepackage{hyperref}
\usepackage{multirow}
\usepackage{mathrsfs} %

\usepackage{algorithm}
\usepackage{algorithmicx}
\usepackage{algpseudocode}
\usepackage{flushend}

\usepackage{ifthen}
\DeclareMathAlphabet{\mbb}{U}{bbold}{m}{n} %
\newcommand \mb[1] {\ifthenelse{\equal{#1}{0}}{\mbb{0}}{\ifthenelse{\equal{#1}{1}}{\mbb{1}}{\mathbb{#1}}}} %

\renewcommand{\l}{\left}

\newtheorem{thm}{Theorem}
\newtheorem{prop}[thm]{Proposition}
\newtheorem{cor}[thm]{Corollary}
\newtheorem{rem}{Remark}
\newtheorem{lem}[thm]{Lemma}
\newtheorem{ass}{Assumption}
\newtheorem{defn}{Definition}
\newtheorem{prob}{Problem}

\newtheorem{exmp}{Example}%

\renewcommand{\textcolor}[2]{#2}
\begin{document}
\begin{frontmatter}

\title{$L_1$ Minimum Energy Estimation of \\ Positive Systems}

\author[inst1]{Alba Gurpegui Ramón}
    \ead{alba.gurpegui_ramon@control.lth.se}
    
\affiliation[inst1]{organization={Department of Automatic Control, Lund University},%
            addressline={Ole Römers väg 1}, 
            city={Lund},
            postcode={223 63}, 
            country={Sweden}}

\begin{abstract}
We present a novel class of minimum energy estimation problems for
positive linear systems with linear costs subject to coupled process
and measurement disturbances, where the latter is bounded by the state and the former is bounded by the noise. Through dynamic programming we show that the problem admits explicit solutions. Remarkably, the Luenberger observer structure of the optimal estimator emerges from the optimization itself, with a switching gain that inherits the sparsity of the disturbance constraints, and it is obtained from a differential equation in the finite horizon case. In the infinite horizon case, an algebraic equation characterizes the limit of the finite-horizon solution and yields a stationary estimator with a static gain. We introduce two distinct notions of detectability to characterize when the static gain stabilizes the error dynamics while preserving their positivity, and analyze under which conditions the stationary solution can be formulated as a linear program. We further establish that a disturbance perturbation of finite $L_1$-norm yields an estimation error of finite $L_1$-norm. The framework is illustrated on a large-scale irrigation network.
\footnotetext{
This work is partially funded by the Wallenberg AI, Autonomous Systems and Software Program (WASP) funded by the Knut and Alice Wallenberg Foundation and the European Research Council (ERC) under grant agreement No 101199738.}
\end{abstract}


\end{frontmatter}

\begin{abstract}
We present a novel class of minimum energy estimation problems for
positive linear systems with linear costs subject to coupled process
and measurement disturbances, where the latter is bounded by the state and the former is bounded by the noise. Through dynamic programming we show that the problem admits explicit solutions. Remarkably, the Luenberger observer structure of the optimal estimator emerges from the optimization itself, with a switching gain that inherits the sparsity of the disturbance constraints, and it is obtained from a differential equation in the finite horizon case. In the infinite horizon case, an algebraic equation characterizes the limit of the finite-horizon solution and yields a stationary estimator with a static gain. We introduce two distinct notions of detectability to characterize when the static gain stabilizes the error dynamics while preserving their positivity, and analyze under which conditions the stationary solution can be formulated as a linear program. We further establish that a disturbance perturbation of finite $L_1$-norm yields an estimation error of finite $L_1$-norm. The framework is illustrated on a large-scale irrigation network.
\end{abstract}
\section{Introduction}
The full state of a dynamical system is rarely available for measurement, and its reconstruction from noisy, partial observations is a fundamental problem in control and signal processing. In minimum energy estimation, the estimate is the state produced by the least-energy disturbance and noise signals consistent with the model and the observed output. Introduced by Mortensen~\cite{Mortensen} and refined by Hijab~\cite{Hijab}, this approach recasts estimation as an optimal control problem whose value function satisfies a Hamilton–Jacobi–Bellman (HJB) equation. It lies between two classical traditions. Stochastic filtering originates in the least squares methods of Legendre and Gauss~\cite{Legendre, Gauss}, continued with the work of Kolmogorov and Wiener~\cite{Kolmogorov, Wiener}, and culminated in the Kalman filter~\cite{Kalman}, with which minimum energy estimation coincides in the linear Gaussian case~\cite{Willems}. Robust estimation, notably $H_{\infty}$ estimation~\cite{Nagpal, ShakedTheodor}, instead protects against the least favorable disturbance without assuming its statistics, at a possible cost in nominal performance~\cite{Simon}. Overviews are given in~\cite{Simon, Kailath}.

The disturbance energy is most often measured by a quadratic cost, which yields a Riccati equation and a linear filter \cite{Goodwin, AndersonMoore}. This quadratic structure is analytically convenient, but the number of unknown parameters in the Riccati equation grows quadratically with the state dimension and the resulting gain is generally dense. In this paper, we instead adopt a linear measure of disturbance energy.
This is consistent with the physical notion of energy as the time integral of power. In an electrical circuit, for instance, power is the product of current $I$ and voltage $V$, giving the total energy $E=\int I^{\top}V\,dt$ \cite[Sec. 1.1]{van2014port}, which is linear in each factor and becomes quadratic only after imposing Ohm's law, $V=RI$. Our cost has an analogous form which pairs each disturbance linearly with a weight, rather than penalizing it quadratically. We refer to its time integral as the disturbance energy. Such a cost is natural when disturbance magnitude accumulates additively over time as in the case of cumulative unknown inputs, persistent nonnegative perturbations, or a total disturbance budget. Linear or $L_1$-induced performance measures have been studied extensively for positive systems, where stability and induced-norm analysis admit exact linear-programming characterizations \cite{Briat, Ebihara}. Positive systems are characterized by the property that their state and output remain nonnegative for any nonnegative input and initial state. This type of dynamics captures a wide range of physical phenomena, and has thus been the subject of substantial control-theoretic research. Classical references on this topic involve the well-established nonnegative matrix theory pioneered by Perron \cite{Perron} and Frobenius \cite{Frobenius}, summarized in standard textbooks \cite{BermanPlemmons, fiedler1986}. A primary advantage of positive systems is that stability can be certified using linear Lyapunov functions \cite{Farina, Blanchini_lyapunov, blanchini_lyapunov2}. This is especially important in the large-scale setting because of the resulting computational scalability \cite{EbiharaPeucelleArzelierTAC2017, tutorial}. 

Estimation for positive systems has been approached in several ways. Positive observers, whose estimates remain nonnegative, were introduced for compartmental systems \cite{vandenHof1998} and characterized by linear programming for general positive systems, where the positivity constraint is shown to be restrictive, as an unstable positive system cannot be stabilized by a positive observer \cite{AitRami2011}. Building on the monotonicity of positive dynamics, interval observers instead track the state through certified upper and lower bounds with positive error dynamics~\cite{Khan2021, ZhangShen2023} and it was recently shown that structuring a positive Luenberger observer as monotonically converging bounds yields stabilizing observer feedback for a broader class of positive systems than earlier designs~\cite{OhlinRantzerTegling2026}. A separate line of research minimizes the $L_1$-induced gain from the disturbance to the estimation error, posing the observer design as a linear program through linear Lyapunov functions \cite{ChenFilter}. These approaches either constrain the observer to be positive or seek to bound its disturbance-to-error gain, rather than deriving the estimator from an optimality principle. In contrast, we formulate the estimation problem as a minimum-energy optimal control problem with linear costs, from which both the observer and the optimal disturbance are obtained explicitly. Crucially, we do not require the observer itself to be positive; instead, the structure of the problem ensures the nonnegativity of the state and the estimation error.

This work is inspired by the linear regulator framework for positive systems, for which the Bellman equation admits an explicit solution whose number of unknowns grows linearly with the state dimension \cite{TAC_paper, Rantzer2022, yuchao}, and the optimal feedback policy inherits its sparsity from the problem constraints. In this paper we develop an estimation counterpart of this line of work, an $L_1$-minimum energy estimation problem, in which the measurements processed by the estimator are corrupted by state-dependent disturbances, and the process disturbance and measurement noise satisfy linear elementwise bounds. 

The main contributions of this manuscript are as follows:
\begin{enumerate}[C1]
    \item We introduce the $L_1$-minimum energy estimation ($L_1$-MEE) problem for positive systems, characterized by a linear time-invariant positive system and linear costs on the process disturbance and the measurement noise, which are state-dependent, the process disturbance being elementwise bounded by the measurement noise and the measurement noise by the state. 

    \item We show that the finite-horizon $L_1$-MEE problem is solved by a backward differential equation, which yields a Luenberger observer with a switching gain (Section~\ref{subsec_finite_hor}).

    \item In the infinite horizon, we prove that the costate of the finite-horizon differential equation converges as the horizon grows if and only if an algebraic equation has a minimal nonnegative solution, which can be computed by value iteration. This solution yields a stationary estimator with a static gain. We further give a  sufficient condition under which the stationary estimator is optimal (Section~\ref{subsec_inf_hor}).

    \item We introduce positive zero-state detectability and H-detectability, and use them to give \textit{a posteriori} and \textit{a priori} conditions under which the stationary gain stabilizes the error dynamics while preserving their positivity (Section~\ref{Stab_detect_section}).
    
    \item We give a linear-programming formulation of the algebraic equation, analyze its primal–dual structure, and relate boundedness of the program to the existence of a stabilizing observer gain (Section~\ref{Section_LP}).
    
    \item We analyze the robustness of the stationary estimator to deviations of the realized disturbance from the disturbance associated with its gain, showing that a perturbation of finite $L_1$-norm keeps the estimation error bounded-input bounded-output stable in the $L_1$ sense (Section~\ref{section_BIBO}). 
    
    \item We illustrate the framework on a large-scale irrigation network, demonstrating the scalability achieved by the linear cost and the positivity structure in Section~\ref{sect_exmpl_scal}.
\end{enumerate}
Although the proposed framework shares structural similarities with the linear regulator problem for positive systems~\cite{TAC_paper}, the estimation problem is not obtained by applying the regulator results to the error dynamics. The disturbance class introduced in C1 bounds the process disturbance in proportion to the measurement noise and the measurement noise by the state. This coupling has a natural interpretation, since in the Kalman filter the gain also grows with the ratio of process noise to measurement noise. In C2, the state is unknown and the optimization is over disturbances consistent with the measurements. The value function is  affine rather than linear, with an extra term driven by the measured data. The optimal estimator turns out to be a Luenberger observer whose gain emerges from the optimal reconstruction of the disturbance rather than being designed. The same Luenberger structure, with a static gain, is retained by the stationary estimator of C3, obtained from an algebraic equation that can be computed by value iteration and, as in the regulator case, admits a linear programming formulation (C5) under detectability conditions (C4).  Finally, since the realized disturbance generally differs from the reconstructed one, C6 establishes a robustness property specific to estimation.
\vspace{-3mm}
\subsection{Notation} Let $\mathbb{R}_{+}$ denote the set of nonnegative real numbers. The inequality $X > Y$ $(X \geq Y)$ means all elements of the matrix $X-Y$ are positive (nonnegative). The notation $\left | X \right |$ means elementwise absolute value. $\mathbb 1$ and $\mathbb 0$ denote the column vectors with all entries equal to one and zero respectively. For a square matrix $X$, the spectral abscissa $\alpha(X)$ is the largest real part of its eigenvalues, and $X$ is Hurwitz if $\alpha(X)<0$. A square matrix is Metzler if all its off-diagonal entries are nonnegative. For $x \in \mathbb R$, $\lceil x \rceil$ denotes the smallest integer greater than or equal to $x$. The signum of a scalar is defined as the set-valued map 
\begin{align*}
    \operatorname{sign}(x) = 
    \scalebox{.9}{\ensuremath{
    \begin{cases}
        \{+1\} & \text{if } x > 0\\
        \{-1\} & \text{if } x < 0\\
        [-1,+1] & \text{if } x = 0
    \end{cases}}}
\end{align*}
and is extended to vectors in an element-wise fashion.
\section{Problem Set Up}
Consider the linear time-invariant (LTI) system
\begin{align}\label{cont_dyn}
    \dot x(t)&=Ax(t)+Bu(t)+ w(t)\\
    y(t)&=Cx(t)+v(t) \notag
\end{align}
where $x(t) \in \mathbb R^n_+$ is the state variable, $u(t)\in \mathbb R^l$ the control input, $y(t) \in \mathbb R^m$ is the measured output, and $w(t)\in \mathbb R^n$ and $v(t)\in \mathbb R^m$ are unknown process disturbance and measurement noise, respectively. The matrices $A\in \mathbb R^{n\times n}$, $B \in \mathbb R^{n \times l}$, $C \in \mathbb R^{m\times n}_+$ are known. The disturbances $w$ and $v$ are restricted to the state-dependent admissible sets 
\begin{align}\label{admissible_w}
    \mathcal W(v)&=\big\{w \in \mathbb R^n \hspace{.5mm}:\hspace{1mm} |w|\leq Hv\big\}\\
    \mathcal V(x)&=\big\{v \in \mathbb R^m \hspace{.5mm}:\hspace{1mm} v\geq 0, \hspace{1mm} v\leq Cx \big\}.
    \label{admissible_v}
\end{align}
with $H \in \mathbb R^{n\times m}_+$ a matrix determined by the problem designer.
The constraint $|w|\leq Hv$ ties the admissible process-disturbance magnitude to the measurement noise through the design matrix $H$, so that process disturbances are admitted only in proportion to the sensor noise, and the bound $v\leq Cx$ ties the measurement noise to the state. Deterministic, state-dependent disturbance bounds of this kind are an established modeling device, studied through robust invariance~\cite{SchaichCannon2015}. The matrix $H$ sets, channelwise, how much process disturbance is admissible per unit of measurement noise, and may be chosen small or large to reflect weak or strong coupling. The upper bound $v \leq Cx$  models measurement noise whose magnitude scales with the state, as in range-based sensing where accuracy degrades with distance~\cite{Takeuchi1985}. Since $v\leq Cx$ and $|w|\leq Hv$ with $H,C\geq 0$, the process disturbance also inherits a state-dependent bound $|w|\leq Hv\leq HCx$, so both disturbances are ultimately governed by the state and vanish together as the state vanishes.
\begin{rem}
    The matrix $H$ specifies the topology and strength of the coupling between process and measurement disturbances. Componentwise,
    \begin{align*}
        |w_i(t)|\leq \sum_{j=1}^m H_{ij}v_j(t)
    \end{align*}
    so diagonal, sparse, block-structured, or dense $H$ represent local, grouped, or network couplings, each entry scaling the admissible process disturbance relative to the sensor noise. Since the bound $|w|\leq Hv \leq HCx$ forces $w$ to vanish wherever the state does, it targets state-dependent couplings such as reaction or compartmental networks in which the disturbance magnitude decays with the state. The coupling also has a filtering interpretation. In the classical Kalman filter, the gain grows with the ratio of process noise to measurement noise, so for fixed $H$ a large $w$ with small $v$ would demand a large gain. The bound $|w|\leq Hv$ ties the admissible process disturbance to the sensor noise, preventing $w$ from dominating $v$ by more than the factor $H$ and keeping the gain bounded, $|L|\leq H$. 
\end{rem}
Process disturbances whose magnitude scales with the state are likewise common in positive systems. In stochastic reaction networks and compartmental epidemic models the intrinsic fluctuations are state-dependent, with magnitude set by the species or compartment counts \cite{ReactionNoise, EpidemicNoise}. More generally, state-dependent noise intensity is a classical object in stochastic control~\cite{Wonham1967, McLane1969}. These models are stochastic, driven by a random process with state-dependent intensity. Our formulation is deterministic, it admits any signal whose realization respects the state-dependent bound and is consistent with the output, without prescribing a distribution. 

Several of our results rely on the following assumption.
\begin{ass}\label{ass}
\begin{subequations}\label{cond:ass}
Consider $\psi \in \mathbb R^n$, $\theta \in \mathbb R^m$, $A \in \mathbb R^{n\times n}$, $C\in \mathbb R^{m\times n}_+$,  $H\in \mathbb R^{n\times m}_+$, $x(-T)\in \mathbb R^n_+$ and $ \chi \in \mathbb R^n$. The following holds:
    \begin{align}  
        &A-HC \hspace{1mm} \text{is a Metzler matrix,}\label{cond:ass:A-HC}\\
        &\theta\geq H^{\top}|\psi|,\label{cond:ass:r-E.Ts}\\
        &C\in \mathbb R^{m\times n}_+ , \hspace{1mm} H\in \mathbb R^{n\times m}_+, \hspace{1mm} H \text{ has non-zero columns,} \label{cond:ass:E}\\
        &x(-T)\geq \chi. \label{cond:ass:x(-T)}
    \end{align}
\end{subequations}
\end{ass}
We next formulate the estimation problem. Given the measured output $y(\cdot)$ and input $u(\cdot)$ over a time interval $[-T, 0]$, $T \in \mathbb R_+$, the goal is to reconstruct the state by selecting the process disturbance $w(\cdot)$ that best explains the data under the admissible-set constraints, at minimum weighted cost. A candidate disturbance explains the data when the induced measurement noise is consistent with the output, $v=y-Cx$, and both signals are admissible, $w\in \mathcal W(v)$ and $v \in \mathcal V(x)$. No distribution is assumed on the disturbances. Any signal, be it deterministic or a realization of a stochastic process, is admitted provided it satisfies these two requirements.
\begin{prob}[$L_1$-MEE Problem]\label{problem_L1MEE}
    Let $A\in \mathbb R^{n \times n}$, $B\in \mathbb R^{n \times l}$, $C\in \mathbb R^{m\times n}_+$, $\psi \in \mathbb R^n$, $\theta \in \mathbb R^{m}$. Suppose Assumption~\ref{ass} holds. Given the signals $u(\cdot)$, $y(\cdot)$ on $[-T, 0]$, and the initial state $x(-T)\in \mathbb R^n_+$, find $w, x: [-T, 0]\rightarrow \mathbb R^n$ and $v:[-T, 0]\rightarrow \mathbb R^m$ minimizing 
    \begin{align}\label{cost_cont_fin}
    \int_{-T}^0 \left[ \psi^{\top}w(\tau) + \theta^{\top}v(\tau)\right] d\tau 
    \end{align}
    subject to~\eqref{cont_dyn},~\eqref{admissible_w} and~\eqref{admissible_v} for all $\tau \in [-T, 0]$. The cost~\eqref{cost_cont_fin} is a weighted $L_1$-functional, linear in the disturbances. We refer to~\eqref{cost_cont_fin} as the disturbance energy.
\end{prob}
Solving Problem~\ref{problem_L1MEE} yields an optimal disturbance signal $w^*(\cdot)$. The corresponding state trajectory of~\eqref{cont_dyn} is the reconstruction of $x$ produced by this least-cost disturbance. We denote it $\bar x: [-T, 0]\rightarrow \mathbb R^n$, initialized at $\chi$, and refer to it as the \textit{$L_1$-minimum energy estimator} of $x$. Theorem~\ref{MEE_LR_fin} shows that $\bar x$ is generated by a Luenberger observer. As shown in Remark~\ref{rem-error}, under the optimal disturbance, the resulting error $e(t)=x(t)-\bar x(t)$ is nonnegative, so $\bar x$ is a certified lower reconstruction of the state, in the spirit of the lower estimate of an interval observer. Problem~\ref{problem_L1MEE} is inspired by the minimum-energy estimation problem in~\cite[Sec. 23.2.1]{hespanha}, adapted to linear costs and positive dynamics, and by the optimization-based approach to constrained estimation in~\cite{Goodwin}. Beyond its physical motivation, the linear and state-proportional form of the admissible sets is what makes the problem tractable, since it keeps the cost linear, renders the error dynamics Metzler, and keeps the positive orthant forward-invariant, so that the estimation problem inherits the explicit, linear  structure of the corresponding regulator problem~\cite{TAC_paper}. 
\begin{rem}
    Equation~\eqref{cond:ass:A-HC} ensures that the positive orthant is forward invariant under the system dynamics. Equation~\eqref{cond:ass:r-E.Ts} ensures that  the cost function is bounded from below. Equation~\eqref{cond:ass:E} ensures that $Hv\geq 0$ so that $\big\{ w :|w|\leq Hv\big\}$ is nonempty and well defined for $x \geq 0$ and the column condition excludes measurement channels that do not affect the admissible disturbances. Equation~\eqref{cond:ass:x(-T)} ensures the initial error is nonnegative $e(-T)=x(-T)-\chi \geq 0$. This condition is always satisfiable, since the state is nonnegative, $x(-T)\geq 0$, so that $\chi =0$ is always admissible. A tighter known lower bound on the initial state may be used in its place when available.
\end{rem}
\vspace{-5mm}
\section{Main Results}
\subsection{Finite Horizon}\label{subsec_finite_hor}
We apply dynamic programming to solve Problem~\ref{problem_L1MEE} in finite horizon. Theorem~\ref{MEE_LR_fin} shows that the resulting optimal estimator is a Luenberger observer, whose switching  gain inherits the sparsity of the disturbance constraints, and it gives the optimal disturbance and the optimal cost explicitly. 
\begin{thm}\label{MEE_LR_fin}
    Consider $A \in \mathbb R^{n\times n}$, $B\in \mathbb R^{n\times l}$, $\psi \in \mathbb R^n$, $C\in \mathbb R^{m\times n}_+ , \hspace{1mm} H\in \mathbb R^{n\times m}_+$, $\theta \in \mathbb R^m$ and $\chi \in \mathbb R^n$. Suppose Assumption~\ref{ass} holds. Let $q:[-T, 0]\rightarrow \mathbb R^n_+$ be the unique solution to the backward ordinary differential equation
    \begin{align}\label{ODE}
            -\dot q(t)^{\top}&=q(t)^{\top}A+\theta^{\top}C-|\psi-q(t)|^{\top}HC\\
            q(0)&= 0.\notag
        \end{align}
    For any gain
    \begin{align}
        L(t)&\in-\mathrm{diag}(\mathrm{sign}(\psi -q(t)))H, \label{eq:sol:L_t}
    \end{align}
    denote by $\bar x : [-T, 0] \rightarrow \mathbb R^n$ the solution to the differential equation 
    \begin{align}
        \dot{\bar x}(t)&=A\bar x(t)+Bu(t)+L(t)(y(t)-C\bar x(t)) \label{eq:sol:x}
    \end{align}
    from an initial condition $\bar x(-T)=\chi$. Then $\bar x$ is an $L_1$-minimum energy estimator of $x$. Moreover, define
    \begin{align}
        -\dot c(t)&=(\theta^{\top}+\psi^{\top}L(t))(y(t)-C\bar x(t)) \label{eq:sol:c}\\
        c(0)&=0. \label{eq:sol:cT}
    \end{align}
     The optimal cost of Problem~\ref{problem_L1MEE} is given by $$q(-T)^{\top}(\chi -x(-T))+c(-T),$$ and is attained by the optimal disturbance $w^*(t)=L(t)v(t)$.
\end{thm}
\begin{rem}\label{rem-error}
    The estimator~\eqref{eq:sol:x} is a Luenberger observer. Unlike positive-observer designs, which constrain the estimate $\bar x$ to remain nonnegative, here no positivity is imposed on the observer. What remains nonnegative is the state $x$ and the estimation error  $e(t)=x(t)-\bar x(t)$. In the absence of control $(u=0)$, the state stays nonnegative because $A$ is Metzler and the admissible disturbances satisfy $|w|\leq Hv\leq HCx$. The error dynamics satisfy 
    $$ \dot e(t)= (A-L(t)C)e(t)+w(t)-L(t)v(t)$$ and under the optimal disturbance $w^*(t)=L(t)v(t)$ reduce to the autonomous system $\dot e(t)= (A-L(t)C)e(t)$. Since $|L(t)|\leq H$ and $A-HC$ is Metzler by Assumption~\ref{ass}~\eqref{cond:ass:A-HC}, the matrix $A-L(t)C$ is also Metzler, so with $e(-T)\geq 0$  the error remains nonnegative and decays asymptotically under detectability properties (see Section~\ref{Stab_detect_section}).
\end{rem}
\begin{rem}\label{rem_K}
    Because $\operatorname{sign}(\cdot)$ is set-valued, the right-hand side of~\eqref{eq:sol:L_t} is in general a set of gains, all achieving the optimal cost. Every diagonal $D=\operatorname{diag}(d)$ with
    \begin{align*}
        d_i\in\operatorname{sign}(\psi_i-q(t)_i), \quad i=1,\dots,n,
    \end{align*}
    yields an admissible gain $L(t)=-DH$. They can switch only when $\psi_i-q(t)_i=0$. Different choices of $D$ may produce different estimator trajectories $\bar x$ generated by~\eqref{eq:sol:x}, as Example~\ref{exmp_1} illustrates, yet all produce the same solution $q(t)$ of~\eqref{ODE} and hence the same optimal value.
\end{rem}
To prove Theorem~\ref{MEE_LR_fin} we use the following Lemma.
\begin{lem}\label{lem_monotone}
    Assume that Assumption~\ref{ass} holds. Then the solution to the differential equation~\eqref{ODE} is monotone for all nonnegative terminal conditions. That is, if $q_i$, $i=1,2$ satisfy the differential equations~\eqref{ODE} and $0 \leq q_1(t_q)\leq q_2(t_q)$, at some $t_q \in [-T, 0]$ then $q_1(\tau_q)\leq q_2(\tau_q)$, for all $\tau_q \leq t_q$. Moreover, if $-\dot{q}(t_q)\geq 0$, then $-\dot q(\tau_q)\geq 0$, for all $\tau_q \leq t_q$. In addition, the solution $c$ of~\eqref{eq:sol:c} satisfies $-\dot c(t)\geq 0$ for all $t \in [-T, 0]$ under the optimal disturbance $w^*=L(t)v$.
\end{lem}
\begin{proof} 
Let $-\dot q_i(t_q)=f_q(q_i)$, where $f_q(q)$ is the right-hand side of~\eqref{ODE}. By the triangle inequality it holds that $\frac{\partial f_q}{\partial q}\geq A^{\top}-C^{\top}H^{\top}$, which is Metzler by Assumption~\eqref{cond:ass:A-HC}, hence $\frac{\partial f_q}{\partial q}$ is Metzler. By~\cite[Rem 1.1, Ch. 3]{Smith1995MonotoneDS} the trajectories that satisfy $-\dot q = f_q(q)$ are monotone. 
Similarly, by Assumptions~\eqref{cond:ass:A-HC}, ~\eqref{cond:ass:r-E.Ts} and the triangle inequality it follows that equation~\eqref{ODE} satisfies the inequality
\begin{align*}
   -\dot{q}(t_q)\geq ( A^{\top}-C^{\top}H^{\top})q(t_q),
\end{align*}
which, again by Assumption~\eqref{cond:ass:A-HC}, completes the first part of the proof. Evaluating this inequality at $t=0$ and using $q(0)=0$ gives $-\dot q(0)\geq 0$, and since $-\dot q(t_q)\geq 0$ propagates to all $\tau_q \leq t_q$ we have $-\dot q(t)\geq 0$ for all $t\in [-T, 0]$ so $q$ is nonincreasing forward in time and $q(t)\geq 0$ for all $t\in [-T, 0]$.

Lastly, consider the equation~\eqref{eq:sol:c}. By Assumption~\ref{ass}~\eqref{cond:ass:r-E.Ts} and $y(t)-C\bar x(t) \geq y(t)-C x(t) \geq 0$ for all $t \in [-T, 0]$ it follows that $-\dot c(t) \geq 0$. 
\end{proof}

Next we prove Theorem~\ref{MEE_LR_fin}.

\begin{proof}
Let $V(t,\xi)$ be the value function at state $x(t)=\xi$ and time $t\in[-T,0]$ and the measurements $y(\tau)$ on $\tau\in [t, 0]$
\begin{align*}
    V(t,\xi)=\inf_{w(\cdot)} \int_{t}^0 \big( \psi^{\top}w(\tau)+\theta^{\top}v(\tau)\big)d\tau.
\end{align*}
By dynamic programming theory~\cite{hespanha}, $V$ satisfies the HJB equation
\begin{align}\label{HJB}
   &-\partial_t V(t,\xi)=\min_{|w|\le H(y-C\xi)}\{\psi^\top w+\theta^\top(y-C\xi) \\
   &~~~~~~~~~~~~~~~~~~~~~~~+\partial_\xi V(t,\xi)^\top(A\xi+Bu+w)\}.\notag
\end{align}
Let 
\begin{align}\label{candidate_VF}
    V(t,\xi):=q(t)^{\top}(\bar x(t)-\xi)+c(t)
\end{align}
be our value function candidate, so $\partial_\xi V=-q(t)$ and $\partial_t V=\dot q(t)^{\top}(\bar x(t)-\xi)+q^{\top} \dot{\bar x}(t)+\dot c(t)$. Substituting into~\eqref{HJB}
\begin{align*}
    &-\Big(\dot q(t)^{\top}(\bar x(t)-\xi)+q(t)^{\top}\dot{\bar x}(t)+ \dot c(t)\Big)\\
    &~~= \min_{|w|\leq H(y-C\xi)}\big\{ \psi^{\top}w+\theta^{\top}(y-C\xi)-q^{\top}(t)(A\xi+Bu+w)\big\}\\
         &~~= \min_{|w|\leq H(y-C\xi)}\big\{ (\psi^{\top}- q^{\top}(t))w\big\} +\theta^{\top}(y-C\xi)-q^{\top}(t)A\xi-q^{\top}(t)Bu\\
         &~~=(\theta^{\top}-|\psi^{\top}- q^{\top}(t)|H)(y-C\xi)-q^{\top}(t)(A\xi+Bu)
\end{align*}
This equality can equivalently be rewritten as  
\begin{align*}
    0&=q^{\top}(t)\dot{\bar x}(t)+\dot q^{\top}(t)\bar x(t) +\dot c(t)-q(t)^{\top}Bu+ (\theta^{\top}-|\psi^{\top}- q^{\top}(t)|H)y \\
        &~~~~~~~~~~~~~~~~~~~~~~~~~~~~
        +(-\dot q^{\top}(t)-q^{\top}(t)A -\theta^{\top}C+|\psi^{\top}- q^{\top}(t)|HC)\xi.
\end{align*}
Since this should hold for an arbitrary $\xi$, it follows that
    \begin{align}\label{ODE_proof}
        -\dot q^{\top}(t)=q^{\top}(t)A+\theta^{\top}C-|\psi^{\top}-q^{\top}(t)|HC
    \end{align}
which yields~\eqref{ODE}.
The remaining $\xi$-independent terms then satisfy
\begin{align}\label{aux}
     0&=q^{\top}(t)(\dot{\bar x}-Bu)+\dot q^{\top}(t)\bar x(t) + (\theta^{\top}-|\psi^{\top}- q^{\top}(t)|H)y +\dot c(t).
\end{align}
 Define $D(t)\in \mathrm{diag}(\mathrm{sign}(\psi -q(t)))$ and $L(t)=-D(t)H$. Substituting~\eqref{ODE_proof} in~\eqref{aux}, the following holds
 \begin{align*}
        0&= (\theta^{\top}-(\psi^{\top}- q^{\top}(t)) D(t)H)y +q(t)^{\top}(\dot{\bar x}-Bu)\\
        &~~~~~~~~~-(q^{\top}(t)A+\theta^{\top}C-(\psi^{\top}-q^{\top}(t))D(t)HC)\bar x(t)+\dot c(t)\\
        &=q^{\top}(t)(-A\bar x-Bu+\dot{\bar x}+D(t)Hy-D(t)HC\bar x)\\
        &~~~~~~~~~~~~~+(\theta^{\top}-\psi^{\top}D(t)H)y-(\theta^{\top}-\psi^{\top}D(t)H)C\bar x+\dot c(t)
        \end{align*}
    which gives~\eqref{eq:sol:x} and~\eqref{eq:sol:c}.

Furthermore, the boundary conditions for $q$ and $c$ are determined by the boundary condition $V(0, \xi)=0$ for all $\xi$, because the cost of the integral~\eqref{cost_cont_fin} on $[0, 0]$ is zero. Imposing this on our value function candidate~\eqref{candidate_VF}
\begin{align*}
    0=V(0,\xi)=q(0)^{\top}(\bar x(0)-\xi)+c(0),\ \forall \xi\in\mathbb R^n.
\end{align*}
This yields the boundary conditions $q(0)=0$, $c(0)=0$.

Lastly, observe that the right hand side of~\eqref{ODE_proof} is continuous. Hence, $V(t,\xi)\in C^{1,2}(t,\xi)$ on $[-T,0]\times\mathbb R^n$. It remains to confirm that $w^*$ is optimal among all admissible disturbances and that $V$ therefore attains the optimal cost. Since the value function $V(t,\xi)$ satisfies the HJB equation~\eqref{HJB}, a standard verification argument~\cite{viscosity_fleming} applies for any admissible $w(\cdot)$ on $[-T, 0]$. Note that
\begin{align*}
    \tfrac{d}{dt}V(t,x(t)) = \partial_t V(t,x(t)) -q(t)^{\top}\dot x(t).
\end{align*}
It follows from the HJB equation~\eqref{HJB} that for any admissible $w$
\begin{align*}
    -\partial_t V \le \psi^\top w + \theta^{\top} (y-Cx) - q^\top(Ax+Bu+w),
\end{align*}
with equality at $w=w^*$ and $w^*(t)=L(t)(y(t)-Cx(t))$. Therefore,
\begin{align*}
   -\tfrac{d}{dt}V(t,x(t)) &= -\partial_t V + q^\top(Ax+Bu+w)\\
    &\leq \psi^\top w + \theta^{\top} (y-Cx),
\end{align*}
with equality at $w=w^*$. Integrating from $-T$ to 0 yields
\begin{align*}
    &V(-T, x(-T))-V(0, x(0)) \leq \int_{-T}^0 [\psi^{\top}w(\tau)+\theta^{\top}(y(\tau)-Cx(\tau))]d\tau.
\end{align*}
The boundary condition $V(0, x(0))=0$ gives 
\begin{align*}
    V(-T, x(-T))\leq \int_{-T}^0 [\psi^{\top}w(\tau)+\theta^{\top}(y(\tau)-Cx(\tau))]d\tau
\end{align*}
with equality for $w=w^*$. Hence, $V(-T, x(-T))=q(-T)^{\top}(\chi-x(-T))+c(-T)$ is the optimal cost with optimal disturbance $w^*(t)=L(t)(y(t)-Cx(t))=L(t)v(t)$.
\end{proof}

The next example illustrates the non-uniqueness of the optimal gain. 
\begin{exmp}\label{exmp_1}
\begin{figure}[h]
    \includegraphics[scale=0.47]{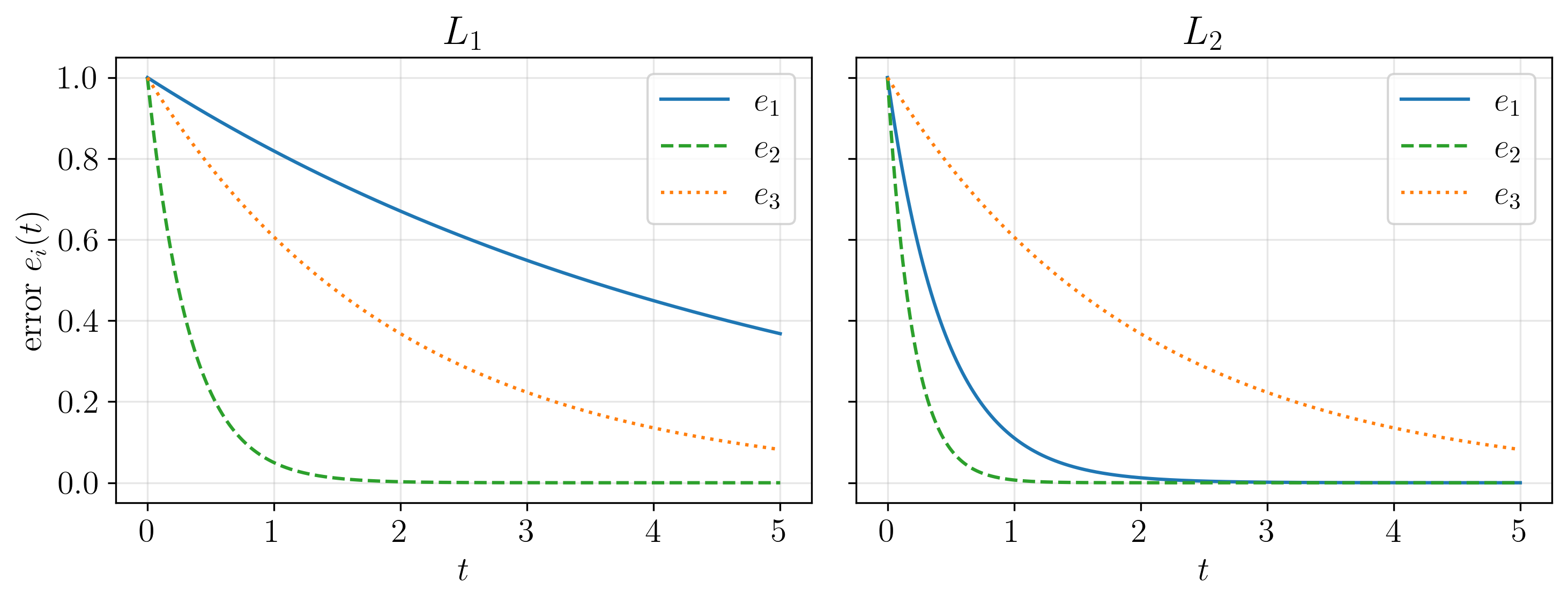}
        
        \caption{Error system trajectories under two optimal gains $L_1$, $L_2$.}
        \label{example1}
    \end{figure}
    Let 
    \begin{align*}
        A&=\begin{bmatrix}
            -1.2&0&0\\
            0&-4&0\\
            0&0&\frac{1}{2}
        \end{bmatrix}, \hspace{1mm} C=I_3, \hspace{1mm} H=I_3\\
        \psi&=\begin{bmatrix}
            0&0&-1
        \end{bmatrix}^{\top}, ~~~~~\hspace{1mm} \theta=\begin{bmatrix}
            0&0&2
        \end{bmatrix}^{\top}
    \end{align*}
    with $B=0$. Assumption~\ref{ass} holds since $A-HC=\mathrm{diag}(-2.2, -5, -1/2)$ is Metzler and $\theta =\begin{bmatrix}
        0&0&2
    \end{bmatrix}^{\top}\geq H^{\top}|\psi|=\begin{bmatrix}
        0&0&1
    \end{bmatrix}^{\top}$. The Perron-Frobenius eigenvalue of $A$ is $\lambda=1/2>0$, so the open loop is unstable. The solution to the ODE~\eqref{ODE} is
    \begin{align*}
        q(t)=\begin{bmatrix}
            0&0&2(1-e^{t/2})
        \end{bmatrix}^{\top}, \hspace{1mm} q(0)=0.
    \end{align*}
    Note that
    \begin{align*}
        \psi - q(t)=\begin{bmatrix}
            0&0&-3+2e^{t/2}
        \end{bmatrix}^{\top}
    \end{align*}
    Hence, both 
    \begin{align*}
        \scalebox{.95}{$L_1=\begin{bmatrix}
            -1&0&0\\
            0&-1&0\\
            0&0&1
        \end{bmatrix}, \hspace{0.5mm} L_2=\begin{bmatrix}
            1&0&0\\
            0&1&0\\
            0&0&1
        \end{bmatrix} \hspace{1mm} \forall t \in [-T, 0]$}
     \end{align*}
    achieve the optimal cost and give the closed loop system matrices
    \begin{align*}
        \scalebox{.8}{$A-L_1 C=\begin{bmatrix}
            -0.2&0&0\\
            0&-3&0\\
            0&0&-\frac{1}{2}
        \end{bmatrix}, \hspace{1mm} A-L_2 C=\begin{bmatrix}
            -2.2&0&0\\
            0&-5&0\\
            0&0&-\frac{1}{2}
        \end{bmatrix}$}
    \end{align*}
    which are Hurwitz. Although $L_1$, $L_2$ incur the same optimal cost, they induce different error transients. Under $L_1$ (Figure~\ref{example1}, left) the first component $e_1$ is the slowest, while under $L_2$ (Figure~\ref{example1}, right) $e_3$ becomes the slowest component instead. 
\end{exmp}
\subsection{Infinite Horizon}\label{subsec_inf_hor}
The infinite-horizon $L_1$-MEE problem is analyzed through the costate of the finite-horizon solution evaluated at the initial time. Theorem~\ref{thm_MEE_LR_inf} shows that this costate converges as $T\rightarrow\infty$ if and only if the algebraic equation~\eqref{ARE} admits a minimal nonnegative solution $q$, which defines a stationary gain and a time-invariant Luenberger observer, the stationary estimator. Since the cost~\eqref{cost_cont_fin} ends at the current time $(t=0)$, the optimal gain may still switch close to $t=0$. Remark~\ref{rem_optimal_inf} describes the resulting optimal estimator, and Proposition~\ref{prop_static_optimal} gives a sign condition on the cost weights under which the stationary estimator is optimal. Finally, Theorem~\ref{Cont_VI} provides a value iteration scheme that computes $q$.
\begin{thm}\label{thm_MEE_LR_inf}
     Consider $A \in \mathbb R^{n\times n}$, $B\in \mathbb R^{n\times l}$, $\psi \in \mathbb R^n$, $C\in \mathbb R^{m\times n}_+ , \hspace{1mm} H\in \mathbb R^{n\times m}_+$, $\theta \in \mathbb R^m$. Let Assumption~\ref{ass} hold. 
     The following are equivalent: 
    \begin{enumerate}[(i)]
        \item The limit $q:=\lim_{T\rightarrow \infty}q_T(-T)$ exists in $\mathbb R^n_+$ where $q_T:[-T, 0]\rightarrow \mathbb R^n_+$ is the solution of the ODE~\eqref{ODE} on $[-T, 0]$.
    \item There exists $q\in \mathbb R^n_+$ satisfying 
    \begin{align}\label{ARE}
        0=q^{\top}A+\theta^{\top}C-|\psi^{\top}-q^{\top}|HC,
    \end{align}
    which is minimal among all nonnegative solutions $\hat q$ of~\eqref{ARE} in the sense that $q \leq \hat q$.
    \end{enumerate}
    Moreover, when $(i)$ and $(ii)$ hold, the optimal gain~\eqref{eq:sol:L_t} of Theorem~\ref{MEE_LR_fin} at the initial time converges, as $T\rightarrow\infty$, to a stationary gain 
    \begin{align}
        L&\in-\mathrm{diag}(\mathrm{sign}(\psi -q))H \label{sol_inf:L}
    \end{align}
    which defines the stationary estimator
    \begin{align}
        \dot{\bar x}(t)&=A\bar x(t)+Bu(t)+L(y(t)-C\bar x(t)), 
        \label{sol_inf:xbar}
    \end{align}
    initialized at $\bar x(-T)=\chi$.
\end{thm}
\begin{rem}\label{rem_optimal_inf}
    The stationary estimator~\eqref{sol_inf:xbar} is not optimal in general. Since the cost~\eqref{cost_cont_fin} ends at the current time, the costate of Theorem~\ref{MEE_LR_fin} satisfies $q(0)=0$ and increases monotonically to $q$ as $t\rightarrow-\infty$ (Lemma~\ref{lem_monotone}). The optimal estimator is therefore the observer~\eqref{eq:sol:x} with the gain~\eqref{eq:sol:L_t}, each $\mathrm{sign}(\psi_i-q_i(t))$ switching at most once, from $\mathrm{sign}(\psi_i-q_i)$ in the past to $\mathrm{sign}(\psi_i)$ near $t=0$. The minimum value of the infinite-horizon $L_1$-MEE problem is
    \begin{align}\label{opt_value_inf}
        \lim_{T\rightarrow\infty}\big[q_T(-T)^{\top}(\chi-x(-T))+c_T(-T)\big]
    \end{align} attained by the optimal disturbance $w^*(t)=L(t)v(t)$.
\end{rem}
\begin{rem}
    The estimation error $e(t)=x(t)-\bar x(t)$ of the stationary estimator satisfies  
    \vspace{-1mm}
    \begin{align}\label{error_inf}
        \dot e(t)=(A-LC)e(t)+w(t)-Lv(t).
    \end{align}
\end{rem}
\begin{rem}
As in the finite horizon case (Remark~\ref{rem_K}), the gain~\eqref{sol_inf:L} is set-valued at any index $i$ with $(\psi-q)_i=0$. Different admissible gains give different estimator trajectories $\bar x$ while $q$ is unchanged. 
\end{rem}
Next we prove Theorem~\ref{thm_MEE_LR_inf}.

\begin{proof}
Let $q_T(t)$ denote the solution to~\eqref{ODE} on $[-T, 0]$ for a fixed horizon $T\in \mathbb R_+$. Throughout the proof we use, without further reference, the monotonicity of $q_T(-T)$ in the horizon $T$ established in Lemma~\ref{lem_monotone}. 

$(i)\Longrightarrow(ii)$ Assume $q_T(-T)$ converges to $q\in \mathbb R^n_+$. By Lemma~\ref{lem_monotone}, $q_T(-T)$ is monotone nondecreasing in $T$, so
\begin{align*}
   \frac{\partial q_T(-T)}{\partial T} \rightarrow 0, \hspace{1mm} \mathrm{as}\hspace{1mm} T \rightarrow \infty
\end{align*}
by monotone convergence. By continuity of the right hand side of~\eqref{ODE} in $q$, the limit $q$ is an equilibrium of~\eqref{ODE}, so the algebraic equation~\eqref{ARE} holds for $q$. By the monotonicity of the dynamics $q_T(t)$, if any $\hat q \geq 0$ satisfies~\eqref{ARE} and is an equilibrium of~\eqref{ODE}, then $q_T(t)\leq \hat q$ for all $t$. Since $q=\lim_{T\rightarrow \infty}q_T(-T)$ it follows that $q\leq \hat q$, and thus $q$ is minimal among all nonnegative solutions of~\eqref{ARE}.

$(ii)\Longrightarrow (i)$ Assume there exists $q\in \mathbb R_+^n$ satisfying~\eqref{ARE} and which is minimal among all nonnegative solutions $\hat q$ of~\eqref{ARE}, in the sense that $q\leq \hat q$. By the monotonicity of the dynamics $q_T(t)$ in Lemma~\ref{lem_monotone} it follows that if any $\hat q \geq 0$ satisfies~\eqref{ARE} and is an equilibrium of~\eqref{ODE}, the trajectory $q_T(t)$ satisfies $q_T(t)\leq \hat q$ for all $t$. Since $q_T(-T)$ is monotone nondecreasing, it follows that $\lim_{T\rightarrow \infty}q_T(-T)\leq \hat q$, and thus $q_T(-T)$ converges. Let $\bar q:= \lim_{T\rightarrow \infty}q_T(-T)$. It follows that $\bar q$ satisfies~\eqref{ARE} and that $\bar q \leq \hat q$ for every nonnegative solution $\hat q$ of~\eqref{ARE}, in particular $\bar q \leq q$. Since $q$ is also minimal, $q \leq \bar q$, hence $\bar q=q$. 
Thus, $q_T(-T)$ converges, which gives $(i)$. 
This establishes the equivalence $(i)\Longleftrightarrow(ii)$. 

By Theorem~\ref{MEE_LR_fin}, the optimal gain of Problem~\ref{problem_L1MEE} on $[-T,0]$ at the initial time is $L_T(-T)\in-\mathrm{diag}(\mathrm{sign}(\psi-q_T(-T)))H$. Since $q_T(-T)\rightarrow q$ by $(i)$, this gain converges to the stationary gain~\eqref{sol_inf:L} as $T\rightarrow\infty$.
\end{proof}

Remark~\ref{rem_optimal_inf} shows that the stationary estimator differs from the optimal one only through the switches of the optimal gain close to the current time. These switches do not occur when the sign of each entry of $\psi-q(t)$ stays constant as $q(t)$ grows from $0$ to $q$. The next proposition shows that, in this case, the stationary estimator is optimal.
\begin{prop}\label{prop_static_optimal}
Suppose $(i)$ and $(ii)$ of Theorem~\ref{thm_MEE_LR_inf} hold, and that for every $i$ either $\psi_i\leq 0$ or $\psi_i\geq q_i$. Then the stationary estimator~\eqref{sol_inf:xbar} is optimal for the infinite-horizon case of Problem~\ref{problem_L1MEE}, with optimal disturbance $w^*(t)=Lv(t)$.
\end{prop}
\begin{proof}
By Lemma~\ref{lem_monotone}, the costate $q(t)$ of Theorem~\ref{MEE_LR_fin} increases monotonically from $q(0)=0$ to $q$ as $t\rightarrow-\infty$. If $\psi_i\leq 0$, then $\psi_i-q_i(t)\leq 0$ for all $t$, and if $\psi_i\geq q_i$, then $\psi_i-q_i(t)\geq \psi_i-q_i\geq 0$ for all $t$. Hence a stationary gain $L\in-\mathrm{diag}(\mathrm{sign}(\psi-q))H$ satisfies $L\in-\mathrm{diag}(\mathrm{sign}(\psi-q(t)))H$ for all $t\leq 0$, so by Theorem~\ref{MEE_LR_fin} it is the optimal gain on every window $[-T,0]$, and therefore on $(-\infty,0]$.
\end{proof}
\begin{rem}
    The suboptimality gap of the stationary estimator is bounded independently of the horizon. Since each entry of the optimal gain switches at most once, the stationary and the optimal gains coincide before the first switch, which  occurs a finite time before the current time $(t=0)$, independently of the horizon. Moreover, this gap vanishes under the condition of Proposition~\ref{prop_static_optimal}, for which no switch occurs.
\end{rem}

The following theorem proposes an iterative fixed-point method to solve the algebraic equation~\eqref{ARE}, following the continuous-time value iteration of~\cite[Thm.~4]{TAC_paper}.

\begin{thm}[Continuous-time value iteration]\label{Cont_VI}
    Let $A \in \mathbb R^{n\times n}$, $\psi\in \mathbb R^n$, $C\in \mathbb R^{m\times n}_+ , \hspace{1mm} H\in \mathbb R^{n\times m}_+$, $\theta\in \mathbb R^m$, and $h \in \mathbb{R}$. 
    Assume $\theta \geq H ^{\top} \left | \psi \right |$, $A-HC$ is Metzler and $A-HC+ h I_n \geq 0$ with $h>0$. 
    Define $\hat{A}=\frac{1}{h}A + I_n$, $\hat{C}=C$, $\hat{H}=\frac{1}{h}H$, $\hat{\psi}=\psi$ and $\hat{\theta}=\frac{1}{h}\theta$. 
    Then, the recursive sequence $\left \{ q_{k} \right \}_{k=0}^{\infty}$ with $q_{0}=0$ and
    \begin{align}\label{p_seq}
        q_{k+1}=\hat{C}^{\top}\hat{\theta}+\hat{A}^{\top}q_k-\hat{C}^{\top}\hat{H}^{\top}\left| q_k - \hat{\psi}\right|
    \end{align}
       has a finite limit if and only if~\eqref{ARE} admits a nonnegative solution, in which case $q_k$ converges to the minimal nonnegative solution $q$ of~\eqref{ARE}.
\end{thm}
\begin{proof}
    Let $F(q):=\hat C^{\top}\hat \theta + \hat A^{\top}q-\hat C^{\top} \hat H^{\top}|q-\hat \psi|$. For $q\leq \tilde q$, $|\tilde q_i-\hat \psi_i|-|q_i-\hat \psi_i|\leq |\tilde q_i- q_i|=\tilde q_i-q_i$. Since $\hat C^{\top} \hat H^{\top} \geq 0$, $F(\tilde q)-F(q)\geq\hat A^{\top}(\tilde q-q)-\hat C^{\top}\hat H^{\top}(\tilde q-q)=(\hat A- \hat H \hat C)^{\top}(\tilde q-q)\geq 0$, using $\hat A- \hat H \hat C \geq 0$ (i.e. $\frac{1}{h}(A- HC +hI )\geq 0$). Moreover, the sequence~\eqref{p_seq} is nondecreasing $q_1-q_0=F(0)=\hat C^{\top}(\hat \theta - \hat H^{\top}|\hat \psi|)\geq 0$ by $\theta\geq H^{\top}|\psi|$ and $C\geq 0$. By monotonicity $q_{k+1}\geq q_k$.
    
    $(\Longleftarrow)$ If $\hat q \geq 0$ solves~\eqref{ARE}, then $F(\hat q)=\hat q$. Since $q_0 \leq \hat q$, monotonicity gives $q_k \leq \hat q$ for all $k$. The sequence is nondecreasing and bounded, hence convergent. By continuity of $F$ its limit $q$ satisfies $F(q)=q$, i.e.~\eqref{ARE}, and $q \leq \hat q$ for every nonnegative solution, so $q$ is the minimal one.

    $(\Longrightarrow)$ If $q_k \rightarrow q \in \mathbb R^n$, continuity of $F$ gives $F(q)=q$, so~\eqref{ARE} has a nonnegative (as $q\geq q_0=0$) solution.
\end{proof}
\subsection{Positive Zero-State Detectability and H-Detectability}\label{Stab_detect_section}
In the context of the infinite horizon $L_1$-MEE problem, stability of the estimator raises two questions. The first is whether the error dynamics can be stabilized by an observer gain satisfying the constraint $|L|\leq H$, which under Assumption~\ref{ass} preserves positivity of the error. The second is whether an unstable estimation error is visible in the performance output, so that it cannot grow without being penalized. This second question is related to \textit{zero-state detectability}~\cite[Def. 3.1]{ZSD}, which requires the state to converge to zero whenever the output is identically zero, in the context of this paper restricted to nonnegative initial states. These two questions lead, respectively, to the following definitions.
\begin{defn}[$H$-detectability]
    Let $A\in\mathbb R^{n\times n}$, $C\in\mathbb R^{m\times n}$ and $H\in\mathbb R^{n\times m}_+$. The pair $(C,A)$ is $H$-\textit{detectable} if there exists an observer gain $L$ with $|L|\leq H$ such that $A-LC$ is Hurwitz.
\end{defn}

\begin{defn}[Positive Zero-State Detectability (PZSD)]\label{PZSD_def}
    Let $M\in\mathbb R^{n\times n}$ be Metzler and $N\in\mathbb R^{1\times n}$ with $N\geq0$. Consider the positive system $\dot x=Mx$, $z=Nx$, and denote by $x(t,x_0)$ its solution with $x(0)=x_0\in\mathbb R^n_+$. The pair $(N,M)$ is \textit{positive zero-state detectable (PZSD)} if $z(t,x_0)=0$ for all $t\geq0$ implies $x(t,x_0)\to0$ as $t\to\infty$.
\end{defn}
Classically, a system is detectable if all unobservable states are stable, which for linear systems is equivalent to the existence of a gain $L$ such that $A-LC$ is Hurwitz~\cite[Thm.~16.8]{hespanha}. For positive systems this equivalence fails once the gain is required to preserve positivity (see Example~\ref{exmp_ZSD_Hdetec}). PZSD only considers nonnegative initial states and admits a positive-systems analogue of the Hautus criterion~\cite{math_detect}. Note that, since $M$ is Metzler, Perron–Frobenius theory implies that $M$ is Hurwitz if and only if it has no nonnegative eigenvector associated with a nonnegative eigenvalue. PZSD requires each such mode to produce a strictly positive output $z$. This is formalized in the following proposition.
\begin{prop}[Prop. 3~\cite{math_detect} ]\label{prop_detect}
    Let $M$ and $N$ be as in Definition~\ref{PZSD_def}. The pair $(N,M)$ is PZSD if and only if $Nv>0$ for every nonnegative eigenvector $v$ of $M$ corresponding to an eigenvalue $\lambda\geq0$.
\end{prop}
It should be noted that PZSD does not imply the existence of a stabilizing admissible gain. In fact, a pair may be PZSD, and fail to be $H$-detectable for every $H$, meaning that no $L$ with $|L|\leq H$ renders $A-LC$ both Metzler and Hurwitz.
\begin{exmp}\label{exmp_ZSD_Hdetec}
Let $$A=\begin{bmatrix}1&1\\1&1\end{bmatrix} \quad \text{and} \quad C=\begin{bmatrix}1&0\end{bmatrix}.$$ The pair $(C,A)$ is observable, since the observability matrix \scalebox{.9}{$\begin{bmatrix}C^{\top}&A^{\top}C^{\top}\end{bmatrix}^{\top}$} has full rank, and hence the pair is classically detectable. For instance, $L=[3,4]^{\top}$ gives the Hurwitz matrix $$A-LC=\begin{bmatrix}-2&1\\-3&1\end{bmatrix},$$ which is not Metzler. Note that, the eigenvalues of $A$ are $\lambda=0,2$, with eigenvectors $[1,-1]^{\top}$ and $[1,1]^{\top}$, respectively. The only nonnegative eigenvector with a nonnegative eigenvalue is $[1,1]^{\top}$, and $C[1,1]^{\top}=1>0$, so $(C,A)$ is also PZSD. However, $[A-LC]_{22}=1$ for every $L$, since the second state is not measured. Since a Hurwitz Metzler matrix has negative diagonal entries~\cite[Lem.~3.1]{Briat_matrices}, $A-LC$ cannot be both Metzler and Hurwitz. Hence no gain preserving positivity of the error stabilizes the error dynamics, and $(C,A)$ is not $H$-detectable for any $H$ satisfying Assumption~\ref{ass}.
\end{exmp}
Under the disturbance $w=Lv$ with $L$ as in~\eqref{sol_inf:L}, let $R=\theta^{\top}+\psi^{\top}L$, so that the running cost is $\ell=Rv$. Both the measured output $y$ and its estimate $\bar y=C\bar x$ are known, and their difference $y-\bar y=Ce+v$ combines the measurement noise with the estimation error. Weighting it by $R$ gives $R(y-\bar y)=\ell +\bar z$ with $\bar z=RCe\geq0$, so that the cost accumulated by the stationary estimator separates into the accumulated noise cost and the accumulated error contribution, $\int_{-\infty}^{0}\ell \,d\tau+\int_{-\infty}^{0}\bar z\,d\tau$. The second term collects the contribution of the estimation error, which makes $\bar z$ the performance output of interest. For instability of the error dynamics to be penalized, every unstable error mode must be visible through $RC$, since otherwise it could grow without affecting $\bar z$. This is precisely what PZSD of $(RC, A-LC)$ guarantees.

The next theorem presents an \textit{a posteriori} condition that determines if the stationary observer gain stabilizes the error dynamics and gives a sufficient condition for the minimum value of the infinite horizon $L_1$-MEE problem to be finite.
\begin{thm}\label{thm_detect_1}
    Suppose the algebraic equation~\eqref{ARE} has a finite solution $q$, and let $L$ satisfy~\eqref{sol_inf:L}. If the pair $((\theta^{\top}+\psi^{\top}L)C, A-LC)$ is PZSD, then $A-LC$ is Hurwitz, $L$ stabilizes the error system under the disturbance $w(t)=Lv(t)$ and $(C,A)$ is $H$-detectable. If, in addition, the measurement noise has finite $L_1$-seminorm $\int_{-\infty}^0\theta^{\top}v(\tau)\,d\tau<\infty$, and the state remains bounded on $(-\infty,0]$, then $c(-T)$ satisfying ~\eqref{eq:sol:cT} is bounded in $T$, and the minimum value of the infinite-horizon problem~\eqref{opt_value_inf} is finite whenever the limit exists.
\end{thm}
\begin{proof} 
    By~\eqref{sol_inf:L}, $L=-\mathrm{diag}(\mathrm{sign}(\psi -q))H$.
    Note that
    \begin{align}\label{aux_L}
        -C^{\top}L^{\top}(\psi -q)=C^{\top}H^{\top}|\psi - q|
    \end{align}
    Substituting this into~\eqref{ARE}, it follows that
\begin{align}\label{aux_thm_detec}
        -(A-LC)^{\top}q=C^{\top}(\theta +L^{\top} \psi ).
    \end{align}
Since $A-HC$ is Metzler and $|L|\leq H$, the matrix $A-LC$
 is Metzler, so its spectral abscissa equals its Perron root $\lambda$, associated with a nonnegative vector $v_p$. Suppose that $A-LC$ is not Hurwitz, i.e. $\lambda= \alpha(A-LC)\geq 0$, with $v_p\geq 0$. Then the right hand side of~\eqref{aux_thm_detec} gives $q^{\top}(-\lambda)v_p =(C^{\top}(\theta+L^{\top}\psi ))^{\top}v_p$. Since $\lambda \geq 0$, $q\geq 0$ and $v_p\geq 0$, $q^{\top}(-\lambda)v_p\leq 0$, hence $(C^{\top}(\theta+L^{\top}\psi ))^{\top}v_p\leq 0$. But by Proposition~\ref{prop_detect} the PZSD property of $((\theta^{\top}+\psi^{\top}L)C, A-LC)$ requires $(\theta^{\top}+\psi^{\top}L)Cv_p>0$. This is a contradiction, so if $((\theta^{\top}+\psi^{\top}L)C, A-LC)$ is PZSD, $A-LC$ is Hurwitz. Consequently, under the disturbance $w^*(t)=Lv(t)$, the error $e(t)=x(t)-\bar x(t)$ satisfies the autonomous dynamics $\dot e(t)=(A-LC)e(t)$, which is asymptotically stable, hence $L$ stabilizes the error system.
 
Let $L_T(t)\in-\mathrm{diag}(\mathrm{sign}(\psi-q_T(t)))H$, $\bar x_T$ and $c_T$ denote the gain, estimator and constant of Theorem~\ref{MEE_LR_fin} on $[-T,0]$, and let $e_T=x-\bar x_T$. Under $w^*=L_Tv$, we have $e_T\geq 0$ (Remark~\ref{rem-error}), $y-C\bar x_T=v+Ce_T\geq 0$ and $0\leq\theta^{\top}+\psi^{\top}L_T\leq 2\theta^{\top}$. Hence,
$$0\leq c_T(-T)\leq 2\int_{-T}^0 \theta^{\top} v\, d\tau + 2\int_{-T}^0 \theta^{\top}Ce_T\, d\tau.$$
The first term is finite by hypothesis. For the second, since~\eqref{ODE} is autonomous with $q_T(0)=0$ and $q_T$ is monotone (Lemma~\ref{lem_monotone}), $L_T$ can be chosen equal to $L$ on $[-T,t^*]$, with $t^*$ independent of $T$. On this interval $e_T(t)=e^{(A-LC)(t+T)}e_T(-T)$, and since $A-LC$ is Metzler and Hurwitz, $\int_{-T}^{t^*}e_T\,dt\leq -(A-LC)^{-1}e_T(-T)$. On the finite interval $[t^*,0]$, $e_T$ remains bounded. Since $e_T(-T)=x(-T)-\chi$ is bounded, $c_T(-T)$ is bounded in $T$, then the minimum value of the infinite horizon case of Problem~\ref{MEE_LR_fin} is finite whenever the limit exists.
\end{proof}
\begin{rem}
In quadratic minimum-energy estimation~\cite[Sec.~23.2.1]{hespanha}, positive definite weights make the problem well posed, while the optimal gain is stabilizing if $(A,C)$ is detectable and $(A,\bar B)$ is controllable~\cite[Thm.~23.2]{hespanha}, where $\bar B$ is the matrix through which the process noise enters. With $\bar B=I_n$, the controllability condition holds trivially. Similarly, in the $L_1$-MEE problem, $\theta\geq H^{\top}|\psi|$ makes the problem well posed, since $\psi^{\top}w+\theta^{\top}v\geq(\theta-H^{\top}|\psi|)^{\top}v\geq 0$ for all admissible disturbances, which requires $\theta\geq 0$ but allows $\psi$ of either sign. By Theorem~\ref{thm_detect_1} a finite solution of~\eqref{ARE} together with PZSD of the closed-loop pair guarantees a stabilizing stationary gain, in which case $(C,A)$ is $H$-detectable.
\end{rem}

The next corollary gives a simple but restrictive \textit{a priori} condition that implies PZSD for all observer gains resulting from Theorem~\ref{thm_MEE_LR_inf}, while a more sophisticated condition is given in Theorem~\ref{detec_apriori}.

\begin{cor}
    Suppose $(\theta-H^{\top}|\psi|)^{\top}C>0$. Then $((\theta^{\top}+\psi^{\top}L)C, A-LC)$ is PZSD for any $L$ satisfying $|L|\leq H$. Moreover, if $(i)$, $(ii)$ in Theorem~\ref{thm_MEE_LR_inf} hold and $L$ satisfies~\eqref{sol_inf:L}, then $A-LC$ is Hurwitz.
\end{cor}
\begin{proof}
    Since $(\theta+L^{\top}\psi )^{\top}C\geq (\theta-H^{\top}|\psi|)^{\top}C>0$, the row vector $(\theta^{\top}+\psi^{\top}L)C$ is strictly positive, so  $(\theta^{\top}+\psi^{\top}L)Cv>0$ for every eigenvector $v \geq 0$, $v\neq 0$ of $A-LC$. Hence $((\theta^{\top}+\psi^{\top}L)C, A-LC)$ is PZSD by Proposition~\ref{prop_detect}. The remainder follows from Theorem~\ref{thm_detect_1}.
\end{proof}

The following theorem shows that PZSD of $((\theta^{\top}+\psi^{\top}L)C, A-LC)$ holds whenever the algebraic equation admits a strictly positive solution, that is, whenever every nonzero state is assigned a positive cost, analogously to the observability assumption in semilinear dynamic programming~\cite[Assumption~2.1(d), Prop.~2]{yuchao_semilinear} and provided the spectral abscissa of $A-LC$ is nonzero.

\begin{thm}\label{thm_detect_2}
    Assume conditions $(i)$ and $(ii)$ in Theorem~\ref{thm_MEE_LR_inf} hold. Let $L $ satisfy~\eqref{sol_inf:L}. If the minimal solution $q$ of $(ii)$ satisfies $q>0$ and the spectral abscissa $\alpha(A-LC)\neq 0$, then the pair $((\theta^{\top}+\psi^{\top}L)C, A-LC)$ is PZSD and the gain $L$ stabilizes the error system~\eqref{error_inf} under the disturbance $w^*(t)=Lv(t)$. Moreover, if $\alpha(A-LC)=0$ then the pair $((\theta^{\top}+\psi^{\top}L)C, A-LC)$ is not PZSD.
\end{thm}
\begin{proof}
    Suppose there exists a vector $q\in \mathbb R^n_+$ solving~\eqref{ARE} with $q>0$. Recall~\eqref{aux_L}. Substituting this into~\eqref{ARE}, it follows that~\eqref{aux_thm_detec}.
    If $A-LC$ is Hurwitz, every trajectory of $\dot e = (A-LC)e$ converges to zero, hence, the pair $((\theta^{\top}+\psi^{\top}L)C, A-LC)$ is PZSD by Definition~\ref{PZSD_def} and $L$ stabilizes the error system~\eqref{error_inf} under optimal disturbance $w^*(t)=Lv(t)$.

    Suppose $A-LC$ is not Hurwitz. Since $A-HC$ is Metzler and $|L|\leq H$, also $A-LC$ is Metzler. Therefore, $A-LC$ has a nonnegative Perron eigenvector $v_p \geq 0$, $v_p\neq 0$, of $A-LC$ associated with the Perron root $\lambda= \alpha(A-LC)\neq 0$ satisfying $(A-LC)v_p=\lambda v_p$. The right hand side of~\eqref{aux_thm_detec} can be written as
    \begin{align}\label{30}
        q^{\top}(-(A-LC))v_p=q^{\top}(-\lambda)v_p.
    \end{align}
    Since $q>0$, $v_p\geq 0$, and $\lambda >0$ it follows that $q^{\top}(-\lambda)v_p<0$. This contradicts the left hand side of~\eqref{aux_thm_detec}, $(C^{\top}(\theta+L^{\top}\psi ))^{\top}v_p \geq 0$. Therefore if $q>0$ and $\lambda \neq 0$ the inequality cannot hold. Consequently, if there exists a solution $q>0$ to~\eqref{ARE} and $\alpha(A-LC)\neq 0$, the pair $((\theta^{\top}+\psi^{\top}L)C, A-LC)$ is PZSD and $L$ stabilizes the error system~\eqref{error_inf} under optimal disturbance $w^*(t)=Lv(t)$.
    
    Finally, under the same hypothesis $q>0$, if $\alpha(A-LC)=0$, the Perron root is $\lambda=0$ with $v_p\neq 0$. Hence,~\eqref{30} gives $q^{\top}(-\lambda)v_p=(C^{\top}(\theta+L^{\top}\psi ))^{\top}v_p=0$. Thus, $((\theta^{\top}+\psi^{\top}L)C, A-LC)$ is not PZSD by Proposition~\ref{prop_detect}.
\end{proof}

The next theorem gives an \textit{a priori} PZSD condition, valid for every observer gain from Theorem~\ref{thm_MEE_LR_inf}, based on the same reasoning as~\cite[Thm.~9]{TAC_paper}. 

\begin{thm}\label{detec_apriori}
    Assume $(i)$-$(ii)$ of Theorem~\ref{thm_MEE_LR_inf}. Let $h\in \mathbb R$ such that $A-HC+h I_n\geq 0$ and $h >0$. Suppose also that
    \begin{align}\label{apriori_detect}
        (\theta^{\top}-|\psi^{\top}|H)C\sum_{i=0}^{n-1} \big(\frac{1}{h}(A-HC)+I \big)^i>0
    \end{align}
    Then $((\theta^{\top}+\psi^{\top}L)C, A-LC)$ is PZSD for any $L$ satisfying $|L|\leq H$.
\end{thm}
\begin{proof} Let $M:=A-LC$, $N:=(\theta^{\top}+\psi^{\top}L)C$, and $\hat M:=A-HC$, $\hat N:=(\theta^{\top}-|\psi^{\top}|H)C$. Since $|L|\leq H$ and $C\geq 0$, we have $M \geq \hat M$ and $N \geq \hat N \geq 0$, the latter by $\theta \geq H^{\top}|\psi|$, and $\frac{1}{h}M+I \geq \frac{1}{h}\hat M+I\geq 0$ by hypothesis. For any nonnegative eigenvector $v \geq 0$, $v\neq 0$ of $M$ with eigenvalue $\lambda \geq 0$,
\begin{align*}
    N\sum_{i=0}^{n-1}\Big( \frac{1}{h}M+I\Big)^i v&= Nv \sum_{i=0}^{n-1}\Big( \frac{\lambda}{h}+1\Big)^i\\
    &\geq \hat N \sum_{i=0}^{n-1}\Big( \frac{1}{h}\hat M + I\Big)^i v >0
\end{align*}
where the equality uses that $v$ is an eigenvector, the middle inequality the entrywise bounds $N \geq \hat N$, $\frac{1}{h}M+I\geq \frac{1}{h}\hat M +I \geq 0$ with $v \geq 0$, and the last inequality uses~\eqref{apriori_detect} with $v \geq 0$, $v \neq 0$. Since $\sum_{i=0}^{n-1} \big(\frac{\lambda}{h}+1 \big)^i\geq n >0$, this gives $(\theta^{\top}+\psi^{\top}L)Cv=Nv>0$. That the pair $((\theta^{\top}+\psi^{\top}L)C, A-LC)$ is PZSD then follows from Proposition~\ref{prop_detect} exactly as in the proof of~\cite[Thm. 9]{TAC_paper} with the output matrix $s^{\top}-r^{\top}K$ replaced by $(\theta^{\top}+\psi^{\top}L)C$. 
\end{proof}
\subsection{Linear program formulation}\label{Section_LP}
Analogous to the regulation case~\cite{TAC_paper}, it is possible to obtain a solution to the algebraic equation~\eqref{ARE} through the linear program
\begin{align}\label{primal_LP}
    \mathrm{Maximize}\quad & 
    \mathbb 1^{\top} q
    \hspace{1mm} \text{over} \hspace{1mm} q\in \mathbb R^n_+, \hspace{1mm} \zeta \in \mathbb R^n_+\notag\\
    \mathrm{Subject \hspace{0.5mm}to}\quad & 
    C^{\top} H^{\top} \zeta \le C^{\top} \theta + A^{\top} q
    \\
    &-\zeta \leq \psi -  q \leq \zeta \notag
\end{align}
\begin{lem}\label{Lem_dual_primal}
The following are equivalent:
\begin{enumerate}[(i)]
    \item The primal linear program~\eqref{primal_LP} has a bounded solution.
    \item The dual linear program
    \begin{align}\label{dual_LP}
    &\mathrm{Minimize} \hspace{7mm} 
    \theta^{\top} C e + \psi^{\top}\nu \hspace{1mm} \text{with}\hspace{1mm}  e\in \mathbb R^n_+ \notag
    \\
    &\mathrm{Subject \hspace{0.5mm}to}\hspace{8mm} 
     Ae - \nu \leq - \mathbb 1
    \\
     &\hspace{17mm}-HCe \leq  \nu \leq HCe\notag
    \end{align}
    is feasible.
    \item There exists a $D\in \mathbb R^{n \times n}$ satisfying $-I\leq D\leq I$ such that $A+DHC$ is Hurwitz.
\end{enumerate}
\end{lem}
\begin{proof}
    $(i)\Longleftrightarrow (ii)$ This follows from the weak duality of the linear programs, together with feasibility of the primal. Taking $q=0$ and $\zeta=|\psi|$ renders the primal feasible, since $C^{\top} H^{\top}|\psi|\leq C^{\top}\theta$ by $\theta\geq H^{\top}|\psi|$ and $C \geq 0$. Hence the primal~\eqref{primal_LP} is bounded if and only if the dual~\eqref{dual_LP} is feasible.

    $(ii)\Longrightarrow (iii)$ For any feasible $(e, \nu)$, the constraint $-HCe \leq \nu\leq HCe$ lets us write $\nu=-DHCe$ for some diagonal $D$ with $-I\leq D \leq I$. The remaining constraint $Ae-\nu \leq -\mathbb 1$ then reads $(A+DHC)e\leq -\mathbb 1<\mathbb 0$. Since $A-HC$ is Metzler and $-I\leq D \leq I$, the matrix $A+DHC$ is Metzler. As $e\geq \mathbb 0$, it follows by item 1 and 18 of~\cite[Thm. 5.1]{fiedler1986} that $A+DHC$ is Hurwitz. Thus there exists $-I\leq D \leq I$ with $A+DHC$ Hurwitz.
    
    $(iii)\Longrightarrow (ii)$ By item 2 and 18 of~\cite[Thm. 5.1]{fiedler1986} there exists a vector $v>0$ such that $(A+DHC)v<\mathbb 0$. Let $\alpha=\min_i|((A+DHC)v)_i|$ and take $e=\frac{1}{\alpha}v$ and $\nu=-DHC e$, then $Ae-\nu=(A+DHC)e\leq -\mathbb 1$, $-HCe\leq \nu \leq HCe$ and $e\geq \mathbb 0$.
\end{proof}

The LP formulation therefore serves two purposes. First, it provides a scalable computational method for obtaining the stationary solution. Second, its dual variables characterize when the disturbance constraints admit a stabilizing observer interpretation.

The next theorem shows that, when the linear program~\eqref{primal_LP} is bounded, its optimizer is a solution of the algebraic equation~\eqref{ARE}, so that the linear program provides an alternative to the value iteration scheme presented in Theorem~\ref{Cont_VI}.
\begin{thm}\label{thm_LP1}
    If the linear program~\eqref{primal_LP} has a bounded value, then the optimizer $q$ solves the algebraic equation~\eqref{ARE}.
\end{thm}
\begin{proof}
    Let $q$, $\zeta$ optimize~\eqref{primal_LP}. Then $C^{\top}H^{\top} \zeta=C^{\top}\theta+A^{\top} q$ and $\zeta=|\psi - q|$, hence
    \begin{align*}
        0&=q^{\top}A+\theta^{\top}C-\zeta^{\top}HC\\
        &=q^{\top}A+\theta^{\top}C-|\psi^{\top}-q^{\top}|HC,
    \end{align*}
    so~\eqref{ARE} holds.
\end{proof}

Note that existence of a bounded solution to the LP~\eqref{primal_LP} is not necessary for the existence of a solution to the algebraic equation~\eqref{ARE}. 
\begin{exmp}
    Let
    \begin{align*}
        A&=\begin{bmatrix}
            -1 & 0 & 0\\
            0 & -1 & 0 \\
            0 & 0 & 1
        \end{bmatrix}; \hspace{1mm} C =\begin{bmatrix}
            1 & 0 & 0
        \end{bmatrix}; \hspace{2mm} H= \begin{bmatrix}
            1 \\ 0 \\ 0
        \end{bmatrix}\\
        \psi &=\begin{bmatrix}
            2 & 0 & 0
            \end{bmatrix}^{\top}; ~~~~\hspace{2mm} \theta= 3.
    \end{align*}
    The third state is unstable and unobservable from $C$. Moreover, the pair $((\theta^{\top}+\psi^{\top}L)C, A-LC)$ is not PZSD for any $|L|\leq H$, since $H$ acts only on the first state, so $A-LC$ has the eigenvector $\mathbb{\mathrm{e}}_3$ with eigenvalue $\lambda=1$   and $(\theta^{\top}+\psi^{\top}L)C\mathbb{\mathrm{e}}_3=0$ as $C\mathbb{\mathrm{e}}_3=0$. The parameters satisfy all conditions in Assumption~\ref{ass}, $A-HC=\mathrm{diag}(-2, -1, 1)$, $3=\theta\geq H^{\top}|\psi|=2$, $H,C\geq 0$. The equation~\eqref{ARE} has a solution
    \begin{align}
        q=\begin{bmatrix}
            2.5 & 0 & 0
        \end{bmatrix}^{\top}.
    \end{align}
    However, the third entry of the $q$ vector in the linear program is unbounded. Thus~\eqref{ARE} admits a finite solution even though the linear program~\eqref{primal_LP} is unbounded, showing that the boundedness of the LP is sufficient but not necessary for solvability of the algebraic equation.
\end{exmp}
Additionally, the next example illustrates that the stationary gain associated with the minimal solution of~\eqref{ARE} need not be unique, and that only some of the stationary gains stabilize the error dynamics.
\begin{exmp}\label{ex:nonunique_gain} 
Let
    \begin{align*}
        A&= \begin{bmatrix}
            -1 & \frac{1}{2}\\
            \frac{3}{2} & -\frac{1}{2}
        \end{bmatrix}; \hspace{1mm} C= \begin{bmatrix}
            1 & 0 
        \end{bmatrix}; \hspace{1mm} H = \begin{bmatrix}
            1 \\
            \frac{1}{2}
        \end{bmatrix}\\
        \psi &= \begin{bmatrix}
            0\\
            -2
        \end{bmatrix}^{\top}; ~~~~~~\hspace{1mm} \theta =1.
    \end{align*}
    The Perron Frobenius eigenvalue of $A$ is $\lambda \approx 0.151>0$ indicating that the open loop system is unstable. $A-HC$ is Metzler and $\theta -H^{\top}|\psi|=0$. The algebraic equation~\eqref{ARE} has the solution $q=\begin{bmatrix}
        0&0
    \end{bmatrix}^{\top}$. Then the first entry of $\psi - q$ is zero so $\mathrm{sign}(\psi -q)_1 \in [-1, 1]$ is set-valued, and the corresponding entry of $L$ can be any value in $[-H_1, H_1]$. Choosing $\text{sign}(\psi_1-q_1)=-1$ gives $L=H$ which gives  
    \begin{align*}
       A-LC= \begin{bmatrix}
            -2 & \frac{1}{2} \\
            1 & -\frac{1}{2}
        \end{bmatrix}
    \end{align*}
which is Hurwitz with spectral abscissa $\alpha(A-LC)=-0.22$. So this $L$ stabilizes the error system. Choosing $L=\begin{bmatrix}
    -1&1/2
\end{bmatrix}^{\top}$ gives
\begin{align*}
       A-LC= \begin{bmatrix}
           0 & \frac{1}{2} \\
            1 & -\frac{1}{2}
        \end{bmatrix}
    \end{align*}
which has an eigenvalue $0.5>0$, so this $L$ is not stabilizing. Both choices of $L$ are compatible with the same solution of~\eqref{ARE}, but only one stabilizes the error system. This illustrates that the algebraic equation alone does not uniquely determine the stabilizing observer gain. 
\end{exmp}
Theorem~\ref{thm_LP1} shows that a bounded linear program yields a solution of~\eqref{ARE}. The next theorem provides a converse. Under PZSD of the closed-loop pair associated with a solution of~\eqref{ARE}, this solution maximizes the linear program, so that the linear program and the value iteration of Theorem~\ref{Cont_VI} return the same stationary solution.
\begin{thm}\label{thm_q_LP1}
    Suppose $q$ solves the algebraic equation~\eqref{ARE} and $L$ satisfies~\eqref{sol_inf:L}. If the pair $((\theta^{\top}+\psi^{\top}L)C, A-LC)$ is PZSD, then $q$ maximizes the LP~\eqref{primal_LP} with $\zeta=|\psi - q|$.
\end{thm}
\begin{proof}
    By Theorem~\ref{thm_detect_1}, PZSD implies $A-LC$ is Hurwitz, so by Lemma~\ref{Lem_dual_primal} the primal LP~\eqref{primal_LP} is bounded. Taking $q$ in~\eqref{primal_LP} as the solution to~\eqref{ARE} and $\zeta=|\psi - q|$ satisfies the constraints of~\eqref{primal_LP}. If $q$ were not optimal, applying the iteration~\eqref{p_seq} of Theorem~\ref{Cont_VI} to $q$ would result in a $\tilde q$ for which $\mathbb 1^{\top}q \leq \mathbb 1^{\top}\tilde q$~\cite{yuchao}. However, since $q$  solves~\eqref{ARE} we have $\tilde q=q$ and thus $q$ is optimal.
\end{proof}

The next theorem relates boundedness of the linear program to the stabilizing properties of the stationary gain. Since~\eqref{sol_inf:L} is set-valued, boundedness guarantees that at least one stationary gain stabilizes the error dynamics, while PZSD guarantees that all of them do.
\begin{thm}\label{thm_q_LP2}
    If the primal linear program~\eqref{primal_LP} has a bounded solution, then all observer gains $L$ satisfying~\eqref{sol_inf:L} define stationary estimators associated with the same solution of~\eqref{ARE}, and at least one such $L$ stabilizes the error system, i.e. the resulting matrix $A-LC$ is Hurwitz. Moreover, if the pair $((\theta^{\top}+\psi^{\top}L)C, A-LC)$ is PZSD, then every $L$ satisfying~\eqref{sol_inf:L} stabilizes the error system.
\end{thm}
\begin{proof}
    By strong duality, the primal and dual linear program,~\eqref{primal_LP},~\eqref{dual_LP} attain the same optimal value. Let $q^*$, $\zeta^*$, $e^*$, $\nu^*$ be optimal with $\theta^{\top}C e^*+\psi^{\top}\nu^*=\mathbb 1^{\top}q^*$. By optimality, the constraints hold with equality,
    \begin{align*}
        |\psi - q^*|&=\zeta^*, \hspace{1mm} (HC)^{\top}\zeta^*=C^{\top}\theta+A^{\top} q^*, \\
         Ae^*-\nu^*&=-\mathbb 1, \hspace{1mm} |\nu^*|=HCe^*.
    \end{align*}
    Let $D=\mathrm{diag}(\mathrm{sign}(\nu^*))\in \mathbb R^{n \times n}$ such that $\nu^*=-DHCe^*$. If $(e^*)_i=0$ for some $i$, the $i$-th coordinate of $(A+DHC)e^*$ is nonnegative, contradicting $(A+DHC)e^* =-\mathbb 1 <0$. Hence $e^*>0$. Since $H\geq 0$, $HCe^* \geq \mathbb 0.$ The established equalities yield 
    \begin{align*}
        &|\psi - q^*|^{\top}HCe^*=(\zeta^*)^{\top}HC e^*=\theta^{\top}Ce^*+(q^*)^{\top}Ae^*\\
        &=\theta^{\top}C e^*-(q^*)^{\top}\mathbb 1+(q^*)^{\top} \nu^* =-\psi^{\top}\nu^*+(q^*)^{\top} \nu^* \\
        &=-(\psi -q^*)^{\top}\nu^*=(\psi -q^*)^{\top}DHC e^*
    \end{align*}
It follows that equality holds if $D\in \mathrm{diag}(\mathrm{sign}(\psi -q^*))$, i.e. $L=-DH\in -\mathrm{diag}(\mathrm{sign}(\psi -q^*))H$. The choice of $L$ is not unique if there exists an index $i$ such that $(\psi -q^*)_i=0.$ By Lemma~\ref{Lem_dual_primal} applied to the dual optimizer $(e^*, \nu^*)$ the matrix $A+DHC$ with $D=\mathrm{diag}(\mathrm{sign}(\nu^*))$ is Hurwitz. Since $D$ also satisfies the equality condition above, the corresponding $L=-DH$ satisfies~\eqref{sol_inf:L} and stabilizes the error system. 

Finally, since the primal LP~\eqref{primal_LP} is bounded, by Theorem~\ref{thm_LP1} the bounded value solves~\eqref{ARE}. Hence, for every $L$ satisfying~\eqref{sol_inf:L} such that $((\theta^{\top}+\psi^{\top}L)C, A-LC)$ is PZSD Theorem~\ref{thm_detect_1} gives that $A-LC$ is Hurwitz, i.e. every such $L$ stabilizes the error system.   
\end{proof}
\subsection{Bounded-Input Bounded-Output Stability under Disturbance Perturbations}\label{section_BIBO}
The infinite-horizon analysis of Section~\ref{subsec_inf_hor} characterizes the stationary estimator under the disturbance $\hat w=Lv$ associated with its gain, for which the error dynamics $e(t)=x(t)-\bar x(t)$ reduces to the autonomous positive system $\dot e(t)=(A-LC)e(t)$. In practice the realized process disturbance need not coincide with $\hat w$. Modeling errors, unmodeled inflows, or a mismatch between the assumed and the actual disturbance all cause the true disturbance to deviate from $\hat w$. We model such a deviation additively, writing the realized disturbance as
\begin{align}\label{subop_w}
    w(t)=\hat w(t)+\delta(t), \hspace{1mm} \hat w(t)=Lv(t)
\end{align}
where $w \in \mathcal W(v)$. Substituting~\eqref{subop_w} into the error dynamics~\eqref{error_inf}, the estimation error satisfies
\begin{align}\label{subopt_w_error}
    \dot e(t)=(A-LC)e(t)+\delta(t), \hspace{1mm} e(-T)=x(-T)-\chi.
\end{align}
The solution to the differential equation~\eqref{subopt_w_error} is 
\begin{align}\label{sol_e_subop}
    \scalebox{.94}{$e(t)=e^{(A-LC)(t+T)}e(-T)+\int_{-T}^t e^{(A-LC)(t-\tau)}\delta(\tau) d\tau.$}
\end{align}
Unlike the case $\delta=0$, the error is no longer autonomous since it is driven by the perturbation $\delta$. The following result characterizes the bounded-input bounded-output property of this system. In particular, it shows that whenever the stationary gain renders $A-LC$ Hurwitz, as guaranteed under PZSD by Theorem~\ref{thm_detect_1}, and the perturbation $\delta$ satisfies a sufficient condition for finiteness of its $L_1$-norm, the error also has finite $L_1$-norm.
\begin{thm}\label{thm_BIBO}
    Let the realized process disturbance be defined as~\eqref{subop_w} so that the estimation error is given by~\eqref{subopt_w_error}. Suppose the matrix $A-LC$ is Hurwitz, the state is bounded on $(-\infty,0]$ and  
    \begin{align}
        \int_{-\infty}^0 \mathbb 1^{\top} |\delta (\tau)|d\tau <\infty.
    \end{align}
    Then the estimation error has finite $L_1$-norm. 
\end{thm}
\begin{proof}
    Let $M=A-LC$ and $\big\| \cdot \big\|$ be the 1-norm. Since $M$ is Hurwitz, there exist $\kappa, \lambda >0$ such that $\big\| e^{Mt}\big\|\leq \kappa e^{-\lambda t}$ for all $t\geq 0$~\cite[Fact 11.18.8]{bernstein}. Hence, from~\eqref{sol_e_subop} 
    \begin{align*}
    \|e(t)\|&\leq \|e^{M(t+T)}\| \|e(-T)\|+\int_{-T}^t \| e^{M(t-\tau)}\| \|\delta (\tau)\| d\tau \\ 
    &\leq \kappa e^{-\lambda (t+T)}\|e(-T)\|+\int_{-T}^t \kappa e^{-\lambda(t-\tau)}\|\delta (\tau)\|d \tau.
    \end{align*}
    Integrating over $t\in [-T, 0]$, the first term gives $\kappa \|e(-T)\|\int_{-T}^{0} e^{-\lambda (t+T)}dt=\frac{\kappa}{\lambda}(1-e^{-\lambda T})\|e(-T)\|\leq\frac{\kappa}{\lambda}\|e(-T)\|$. Let $u=t-\tau$, the second term gives
    \begin{align*}
    &\int_{-T}^{0} \int_{-T}^t \kappa e^{-\lambda (t-\tau)}\|\delta(\tau)\| d\tau\, dt\\ &~~~~=\int_{-T}^{0}\|\delta(\tau)\|\Big( \int_{\tau}^{0}\kappa e^{-\lambda(t-\tau)}dt\Big)d\tau\\ &~~~~=\int_{-T}^{0}\|\delta(\tau)\| \Big( \int_0^{-\tau}\kappa e^{-\lambda u} du\Big) d\tau\\ 
    &~~~~\leq \frac{\kappa}{\lambda} \int_{-T}^{0}\|\delta(\tau)\|d\tau.
    \end{align*}
    
    Therefore
    \begin{align*}
        \int_{-T}^{0}\mathbb 1^{\top}|e(t)|dt \leq \frac{\kappa}{\lambda}\Big(\|e(-T)\|+\int_{-\infty}^{0}\mathbb 1^{\top}|\delta(\tau)|d \tau \Big)<\infty
    \end{align*}
    so $e$ has finite $L_1$-norm. 
\end{proof}

The condition $\int_{-\infty}^{0} \|\delta(\tau) \|d\tau <\infty$ is not merely a sufficient condition. The following result shows that for a constant perturbation $\delta(t)= \epsilon$ with infinite $L_1$-norm, the error has infinite $L_1$-norm, even though $A-LC$ is Hurwitz. 

\begin{thm}
    Let $A-LC$ in~\eqref{sol_e_subop} be Metzler and Hurwitz and $e(-T)=e_0 \geq 0$ for all $T$. There exist signals $\delta(t)$ for which 
    \begin{align}
        \int_{-\infty}^{0} \mathbb 1^{\top} |\delta (\tau)|\,d\tau = \infty
    \end{align}
    that lead to an unbounded $L_1$-norm of the error~\eqref{sol_e_subop}. 
\end{thm}
\begin{proof}
    Let $M := A-LC$. Assume $\delta(t)=\epsilon$ is a sufficiently small constant signal such that $\epsilon>0$ and $e_0+M^{-1}\epsilon \geq 0$. Then, the error~\eqref{sol_e_subop} becomes    \begin{align}\label{aux_BIBO}
        e(t)&=e^{M(t+T)}e_0+\int_{-T}^t e^{M(t-\tau)}d\tau \hspace{0.5mm}\epsilon\\ 
        &=-M^{-1}\epsilon + e^{M(t+T)}[e_0+M^{-1}\epsilon]
    \end{align}
    By assumption $e_0\geq 0$. Since $M$ is Metzler and Hurwitz, $-M^{-1}$ exists and has nonnegative entries, $e^{M(t+T)} \geq 0$ and $e_0+M^{-1}\epsilon \geq 0$ by assumption, so $e^{M(t+T)}[e_0+M^{-1}\epsilon] \geq 0$ and
    \begin{align*}
        -M^{-1}\epsilon + e^{M(t+T)}[e_0+M^{-1}\epsilon]\geq -M^{-1}\epsilon >0.
    \end{align*}
    Integrating~\eqref{aux_BIBO}
    \begin{align*}
        \int_{-T}^{0} \mathbb 1^{\top}|e(t)|dt&=\int_{-T}^{0} \mathbb 1^{\top} \big(-M^{-1}\epsilon+e^{M(t+T)}[e_0+M^{-1}\epsilon] \big)dt\\ &\geq \int_{-T}^{0} \mathbb 1^{\top} \big(-M^{-1}\epsilon\big) dt.
    \end{align*}
    Since $-M^{-1}\epsilon >0$, then $\lim_{T \rightarrow \infty}\int_{-T}^0 \mathbb 1^{\top}\big(-M^{-1}\epsilon \big) dt = \infty$, which completes the proof.
\end{proof}
\begin{exmp}
\begin{figure}[h]
     \centering
     \includegraphics[scale=0.41]{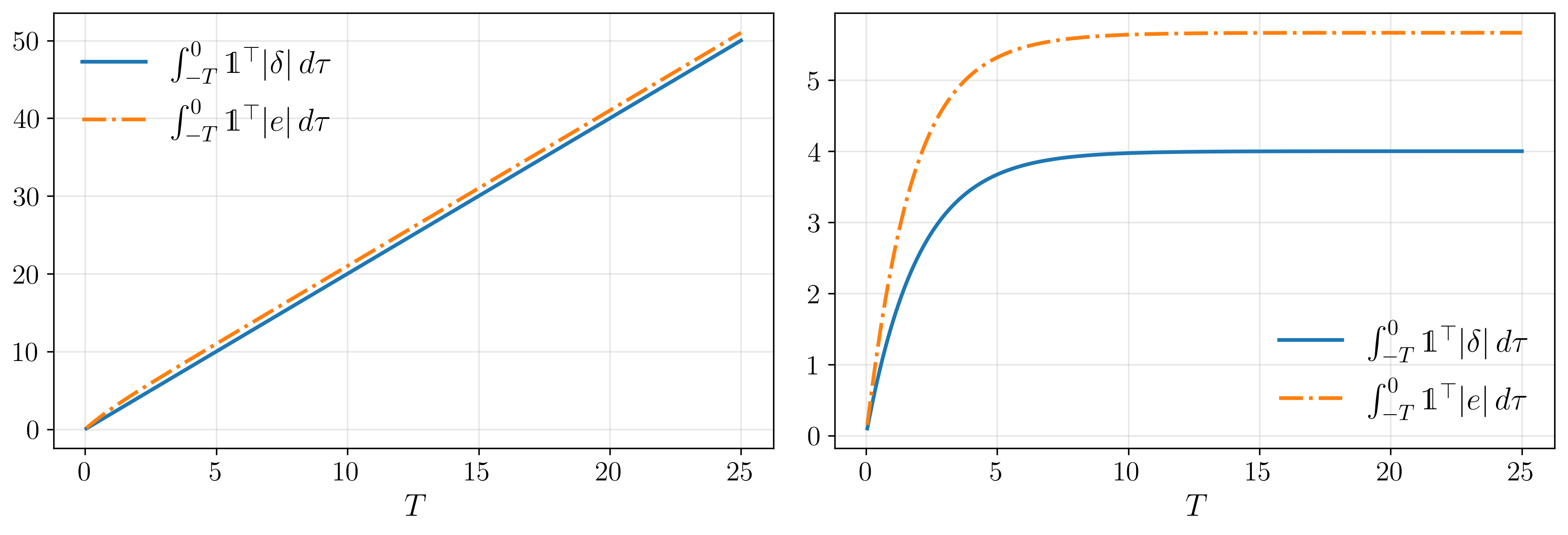}
        \caption{Integrals $\int_{-T}^0\mathbb 1^\top|\delta|\,d\tau$ (solid) and $\int_{-T}^0\mathbb 1^\top|e|\,d\tau$ (dash-dotted) for the two-state example as functions of $T$. Left: constant perturbation $\delta=\epsilon$, both integrals diverge. Right: decaying perturbation $\delta=\epsilon e^{\mu t}$, both integrals saturate, at $4$ and $17/3$ respectively.}
        \label{example4}
    \end{figure}
    Let 
    \begin{align*}
        A=\begin{bmatrix}
            1&1\\
            1&1
        \end{bmatrix}, \hspace{1mm} C=I_2, \hspace{1mm} H= 3 I_2.
    \end{align*}
    The open loop is unstable, since $A$ is Metzler with $\alpha (A)=2 >0$. Choosing $L=3I_2$ (admissible since $|L|\leq H$), $M=A-LC$ becomes 
    \begin{align*}
        M=\begin{bmatrix}
            -2&1\\
            1&-2
        \end{bmatrix}
    \end{align*}
    which is Metzler and Hurwitz with eigenvalues $-3$ and $-1$. Moreover, the inverse is 
    $$-M^{-1}=\tfrac13\begin{bmatrix}2&1\\1&2\end{bmatrix}\ge0.$$
    We compare two admissible perturbations $\delta$, taking the initial error $e(-T)=\tfrac{3}{2} \mathbb 1$ so that $e(-T)+M^{-1}\epsilon=\tfrac{1}{2}\mathbb 1 \geq 0$. Since $\mathbb 1$ is an eigenvector of $M$ with $M \mathbb 1=-\mathbb 1$, we have $e^{Mt}\mathbb 1=e^{-t} \mathbb 1$.
    \begin{itemize}
        \item Constant perturbation $\delta(t)=\epsilon$ with $\epsilon= \mathbb 1$. Here $\int_{-\infty}^0\mathbb 1^{\top}|\delta|d \tau=\int_{-\infty}^0 2 d\tau=\infty.$ By~\eqref{aux_BIBO},
        \begin{align*}
            e(t)&=-M^{-1}\epsilon+e^{M(t+T)}(e_0+M^{-1}\epsilon)\\ &=\mathbb 1+\tfrac12 e^{-(t+T)}\mathbb 1=\big(1+\tfrac12 e^{-(t+T)}\big)\mathbb 1\ge0,
        \end{align*}
        so 
        \begin{align*}
            \mathbb 1^{\top} |e(t)|=\big( 1+\tfrac{1}{2}e^{-(t+T)}\big)\mathbb 1^{\top}\mathbb 1=2+e^{-(t+T)}
        \end{align*}
        and
        \begin{align*}
            &\lim_{T\rightarrow \infty}\int_{-T}^0 \mathbb 1^\top|e(\tau)|\,d\tau\\
            &~~~~~~~~~~~~~=\lim_{T\rightarrow \infty} \int_{-T}^0 \big(2+e^{-(\tau+T)}\big)d\tau=\infty.
        \end{align*}
        \item Decaying perturbation $\delta(t)=\epsilon e^{\mu t}$, with $\epsilon=\mathbb 1$, $\mu=\tfrac{1}{2}$. Here $$\int_{-\infty}^{0}\mathbb 1^{\top}|\delta|d\tau=\int_{-\infty}^{0}2e^{\mu \tau}d\tau =\tfrac{2}{\mu}=4.$$ Since $e(-T)$, $\epsilon$ and $M\mathbb 1=-\mathbb 1$ are all proportional to $\mathbb 1$, the error stays proportional to $\mathbb 1$, writing $e(t)=a(t)\mathbb 1$ and substituting $\dot e=Me+\delta$ gives the equation
        \begin{align}\label{a.2}
            \dot a \mathbb 1=aM \mathbb 1+ e^{\mu t}\mathbb 1=(-a+e^{\mu t})\mathbb 1
        \end{align}
        where $a(-T)=\tfrac{3}{2}$ from $e(-T)=\frac{3}{2}\mathbb 1$. Solving $\dot a(t)=-a(t)+e^{\mu t}$ gives
        \begin{align}\label{a}
            a(t)=\tfrac32 e^{-(t+T)}+\frac{e^{\mu t}-e^{-t-(1+\mu)T}}{1+\mu}\ge 0.
        \end{align}
        Hence, $\mathbb 1^{\top}|e(t)|=2a(t)$, and since 
        \begin{align*}
            \lim_{T\rightarrow \infty}\int_{-T}^{0} \tfrac{3}{2}e^{-(\tau+T)}d\tau=\tfrac32,\\ \lim_{T \rightarrow \infty}\tfrac{1}{1+\mu}\int_{-T}^{0} \big(e^{\mu \tau}-e^{-\tau-(1+\mu)T} \big)d\tau&=\tfrac{1}{\mu(1+\mu)}
        \end{align*}
        so
        \begin{align*}
            \lim_{T\rightarrow \infty} \int_{-T}^0\mathbb 1^\top|e(\tau)|\,d\tau=\lim_{T\rightarrow \infty} 2\int_{-T}^0 a(\tau)\,d\tau=\tfrac{17}{3}.
        \end{align*}
    \end{itemize}
    Figure~\ref{example4} shows the integrals $\int_{-T}^0 \mathbb 1^{\top}|\delta|d\tau$ and $\int_{-T}^0 \mathbb 1^{\top}|e|d\tau$ for both perturbations as functions of $T$. For the constant perturbation (left), $\int_{-T}^0 \mathbb 1^{\top}|\delta|d\tau=2T$ grows without bound and the error integral grows with it, confirming that the error integral diverges despite $M$ being Hurwitz. For the decaying perturbation (right), $\int_{-T}^{0} \mathbb 1^{\top}|\delta| d\tau$ saturates at $\tfrac{2}{\mu}=4$ and the error integral saturates at $17/3$, which matches the closed-form values above and illustrates Theorem~\ref{thm_BIBO}.
\end{exmp}
In the next example we consider an irrigation network of $n$ contiguous sections, each holding a water volume $x_i(t)\geq 0$, connected to its neighbors and monitored by level sensors placed every second section. Only half the sections are instrumented, and the readings over-report in proportion to the true volume. We show that the $L_1$-MEE estimator reconstructs the volumes of all sections, instrumented or not, with error decaying to zero, and that the computation scales to networks of hundreds of sections.
\section{Example: Irrigation network}\label{sect_exmpl_scal}
\begin{figure*}[!ht]
\centering
\begin{tikzpicture}[scale=0.95, transform shape,
  reservoir/.style={draw, thick, minimum width=0.8cm, minimum height=0.6cm,
                    fill=blue!10, rounded corners=1pt},
  gauge/.style={draw, thick, circle, minimum size=0.35cm, fill=orange!30,
                inner sep=0pt, font=\scriptsize},
  arrow/.style={-{Stealth[length=2mm]}, thick},
  lossarrow/.style={-{Stealth[length=2mm]}, thick, red!60!black},
  futurearrow/.style={-{Stealth[length=2mm]}, thick, gray!50, dashed},
  every node/.style={font=\small}
]
\foreach \i in {1,...,7} {
  \node[reservoir] (r\i) at (\i*1.4, 0) {$x_{\i}$};
}
\node at (8*1.4, 0) {$\cdots$};
\node[reservoir] (rn) at (9*1.4, 0) {$x_n$};
\foreach \i/\j in {1/2, 2/3, 3/4, 4/5, 5/6, 6/7} {
  \draw[arrow] (r\i.east) -- node[above, font=\scriptsize] {$\beta$} (r\j.west);
  \draw[arrow] (r\j.west) to[bend left=30] (r\i.east);
}
\draw[arrow] (r7.east) -- (8*1.4 - 0.2, 0);
\draw[arrow] (8*1.4 + 0.2, 0) -- (rn.west);
\foreach \i/\name in {1/$w_1$, 2/$w_2$, 3/$w_3$, 4/$w_4$, 5/$w_5$, 6/$w_6$, 7/$w_7$, 9/$w_n$} {
  \draw[arrow] (\i*1.4, 1.1) node[above] {\name} -- (\i*1.4, 0.32);
}
\foreach \i in {1,2,3,4,5,6,7,9} {
  \draw[lossarrow] (\i*1.4 + 0.25, -0.32) -- ++(0.35, -0.35)
    node[right, font=\scriptsize, red!60!black] {$\gamma$};
}
\foreach \i/\j/\rname in {1/1/r1, 3/2/r3, 5/3/r5, 7/4/r7, 9/m/rn} {
  \node[gauge] (g\j) at (\i*1.4, -1.3) {\j};
  \draw[arrow] (\rname.south) -- ++(0,-0.4) -- (g\j.north)
    node[midway, right, font=\scriptsize] {$y_\j$};
  \node[font=\scriptsize, left=1pt of g\j] {$v_\j$};
}
\end{tikzpicture}
\caption{Irrigation network with $n$ sections. Volumes $x_i$ exchange water bidirectionally with their neighbors at rate $\beta$ and dissipate locally at rate $\gamma$ (red arrows). Disturbances $w_i$ act on every section; gauges placed every $k=2$ sections report measurements $y_j$ corrupted by state-dependent noise $v_j$.}
\label{fig:irrigation_corridor}
\end{figure*}
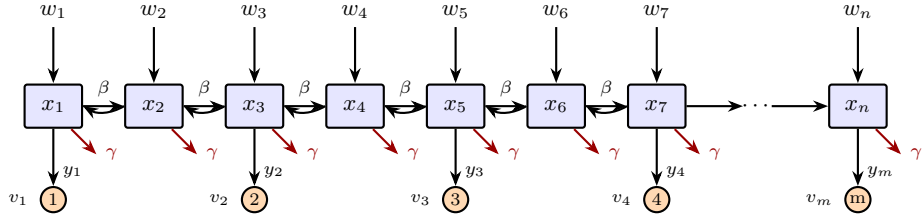
An irrigation network distributes water through a system of channels that supply crops across a cultivated area, with the flow regulated by gates and monitored by level sensors. Keeping the water at appropriate levels along the channels requires knowing how much water each stretch actually holds, but sensors are costly, so only a few points are instrumented and their readings are imperfect. The network is line-shaped and subdivided into $n$ contiguous sections, indexed from $i=1$ to $i=n$, each holding a volume $x_i(t)\geq 0$ of water. The dynamics are governed by three effects. First, each section loses water locally at a rate $\gamma >0$, representing evaporation, infiltration, or leakage through small fissures, which drains section $i$ in proportion to the volume it holds. Second, adjacent sections are connected by open gates, so water is exchanged bidirectionally at rate $\beta >0$ between neighbors, modeling the diffusive mixing through the connecting gates. Third, process disturbances $w_i$ capture unmodeled inflows from pump leakages. A small subset of sections is equipped with gauges, which report the local volume corrupted by state-dependent (over-reading) noise $v_j \in [0, (Cx)_j]$. In this setting sensor inaccuracy scales with the true volume, so a section holding more water is measured less accurately in absolute terms. The task is to reconstruct the volumes in all $n$ sections, including the unmeasured ones, from these few noisy readings. A control input $u$, whose joint design with the estimator is a natural extension paralleling the framework of~\cite{TAC_paper}, is left as future work. The system dynamics corresponding to Figure~\ref{fig:irrigation_corridor} are modeled as
\begin{align}\label{state-space}
    \dot x_1&=-\gamma x_1+\beta x_2+w_1\\
    \dot x_i &=-\gamma x_i + \beta (x_{i-1}+x_{i+1}) +w_i, \hspace{1mm} i=2,\dots, n-1 \notag\\  
    \dot x_n&=-\gamma x_n+\beta x_{n-1}+w_n \notag
\end{align}
Stacking the section volumes into $x \in \mathbb R^n_+$, the sensor readings into $y\in \mathbb R^m_+$, the disturbances into $w \in \mathbb R^n$ and $v \in \mathbb R^m$, the network can be represented as the LTI positive system
\begin{align*}
    \dot x(t)=Ax(t)+w(t), \hspace{2mm} y(t)=Cx(t)+v(t)
\end{align*}
matching the $L_1$-MEE setup of equation~\eqref{cont_dyn} with $u\equiv 0$. The disturbances satisfy $0\leq v \leq Cx$ and $|w|\leq Hv$ as in~\eqref{admissible_w},~\eqref{admissible_v}. Assume the network is homogeneous and level sensors are placed every $k=2$ sections. 
Then the system matrices $A\in \mathbb R^{n \times n}$, $C \in \mathbb R^{m \times n}$ are
\begin{align*}
    \scalebox{.85}{$A= \begin{bmatrix}
        -\gamma& \beta &0 & \cdots & 0\\
        \beta & -\gamma & \beta & \cdots & 0\\
        0& \beta & -\gamma & \ddots & \vdots\\
        \vdots & \vdots& \ddots & \ddots & \beta\\
        0&0&\cdots & \beta & -\gamma
    \end{bmatrix}, \hspace{2mm} C=\begin{bmatrix}
        1&0&0&0&0&\cdots & 0\\
        0&0&1&0&0&\cdots & 0\\
        0&0&0&0&1&\cdots & 0\\
        \vdots & &&&&\ddots&\vdots\\
        0&0&0&0&0&\cdots & 1
    \end{bmatrix}$}
\end{align*}
The matrix $A$ is a tridiagonal Toeplitz matrix, hence its eigenvalues are $-\gamma +2\beta \cos (k\pi / (n+1))$, $k=1\dots n$ and its spectral abscissa is $\alpha(A)=-\gamma+2\beta \cos(\pi / (n+1))$, positive whenever $\gamma < 2\beta \cos(\pi/ (n+1))$. The matrix $H\in \mathbb R^{n\times m}_+$ specifies how much process disturbance each section is allowed, and ties this to the sensor that monitors it. Sections are grouped in consecutive blocks of $k$, one block per sensor, so section $i$ is monitored by sensor $\lceil i/k \rceil$ and receives a disturbance budget $\eta >0$ from it
\begin{align*}
    H_{i,j}=\begin{cases}
        \eta& j = \lceil i/k \rceil\\
        0& \mathrm{otherwise}.
    \end{cases}
\end{align*}
Each row of $H$  has exactly one nonzero entry, in the column of the sensor monitoring that section, and each column has $k$ nonzero entries, one per section in its block. For $k=2$ this means sections 1,2 draw on the first sensor, sections 3, 4 on the second and so on. In particular no column of $H$ vanishes, so Assumption~\ref{ass}\eqref{cond:ass:E} holds. The parameter $\eta$ sets how much unmodeled inflow is admissible per unit of sensor noise, a larger $\eta$ permits larger process disturbances for the same measurement error.  
The cost weights are uniform, $\psi = \rho_{\psi} \mathbb 1_n$ and $\theta = \rho_{\theta} \mathbb 1_m$, where $\rho_{\psi}, \rho_{\theta} >0$ are tunable.  

The choice of $\gamma, \hspace{1mm}\beta, \hspace{1mm}\eta,\hspace{1mm} \rho_{\psi},\hspace{1mm} \rho_{\theta}$ depends on conditions~\eqref{cond:ass:A-HC},~\eqref{cond:ass:r-E.Ts}. Fix $\gamma=0.8, \hspace{1mm}\beta = 0.5,\hspace{1mm} \eta=0.5,\hspace{1mm} \rho_{\psi}=1,\hspace{1mm} \rho_{\theta}=2$ and assume that the network has $n=10$ sections. Under these parameters the Metzler condition holds and the spectral abscissa of $A$ is $\alpha(A)\approx 0.16$, so the open-loop dynamics are unstable. For $n=10$, $C$ has $m=5$ rows and the matrix $H\in \mathbb R^{10 \times 5}$ takes the form
\begin{align*}
   \scalebox{.9}{$ H^{\top}=\begin{bmatrix}
        \eta&\eta&0&0&0&0&0&0&0&0\\
        0&0&\eta&\eta&0&0&0&0&0&0\\
        0&0&0&0&\eta&\eta&0&0&0&0\\
        0&0&0&0&0&0&\eta&\eta&0&0\\
        0&0&0&0&0&0&0&0&\eta&\eta\\
    \end{bmatrix}$}
\end{align*}
\subsubsection{Finite Horizon}
Figure~\ref{fin-hor-1} reports the finite-horizon solution on $[-T,0]$ with
$T=15$, plotted for  $t=\tau+T\in[0,T]$. $q(t)$ decreases monotonically over the horizon, from its stationary value at the initial time to $0$ at the terminal time, consistent with Lemma~\ref{lem_monotone}. The entries of $L(t)$ are piecewise constant, switching only when a component of $\psi - q(t)$ changes sign, and each nonzero entry takes values in $\{-\eta,\eta\}$ as indicated by \eqref{eq:sol:L_t}. The estimator is initialized at an initial estimate $\chi$ drawn componentwise uniformly on $[-2, 2]$, below the true initial volumes $x(-T)$, so that the initial error $e(-T)=x(-T)-\chi \geq 0$ as required by~\eqref{cond:ass:x(-T)}. The true volumes $x_i$ (solid) and estimates $\bar x_i$ (dashed) agree after an initial
transient, and the error $e=x-\bar x$ stays nonnegative and decays, confirming
the positive error dynamics of Remark~\ref{rem-error}.
\begin{figure}[h]
        \centering
        \includegraphics[scale=0.35]{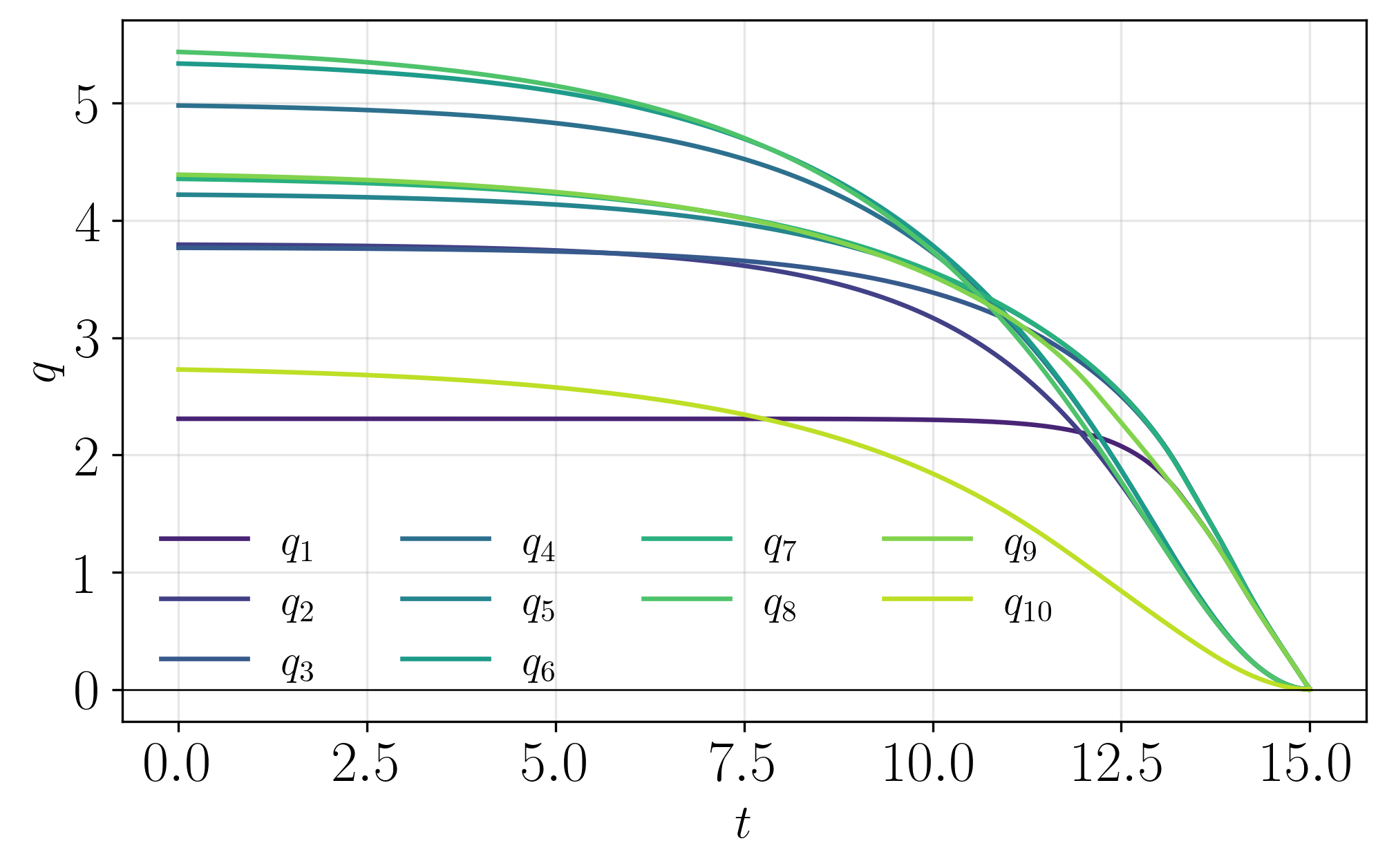}
        \includegraphics[scale=0.35]{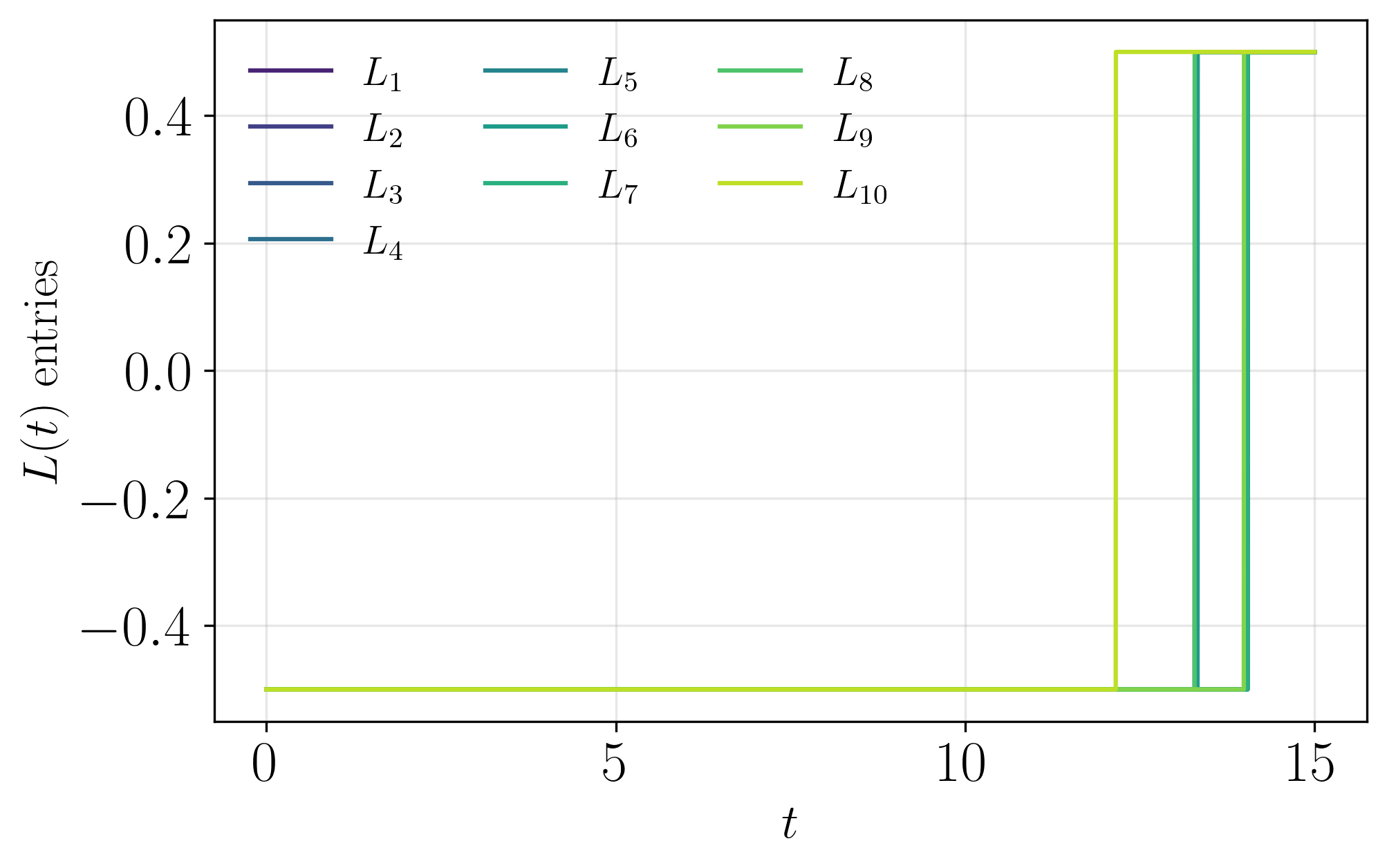}
        \includegraphics[scale=0.35]{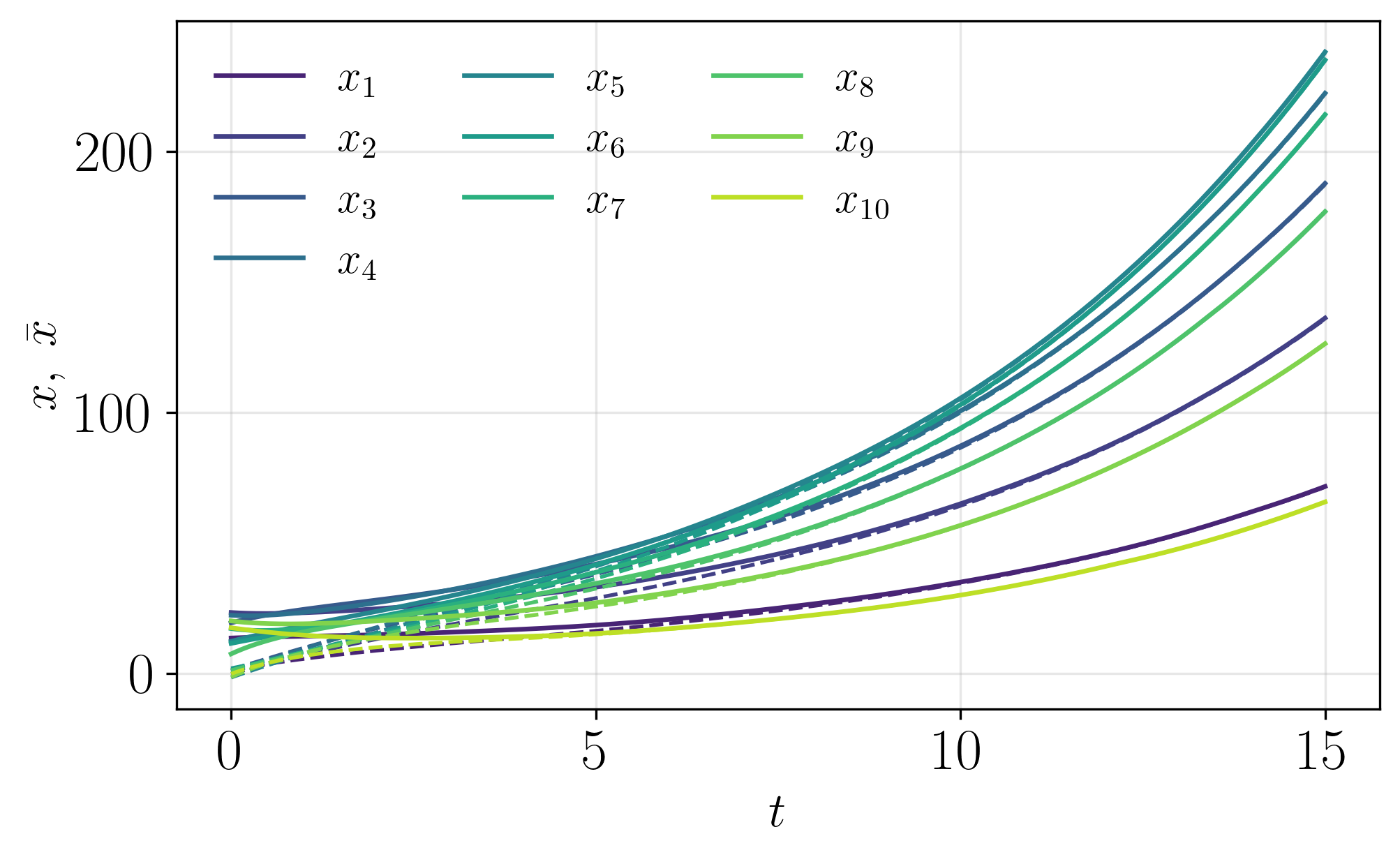}
        \includegraphics[scale=0.35]{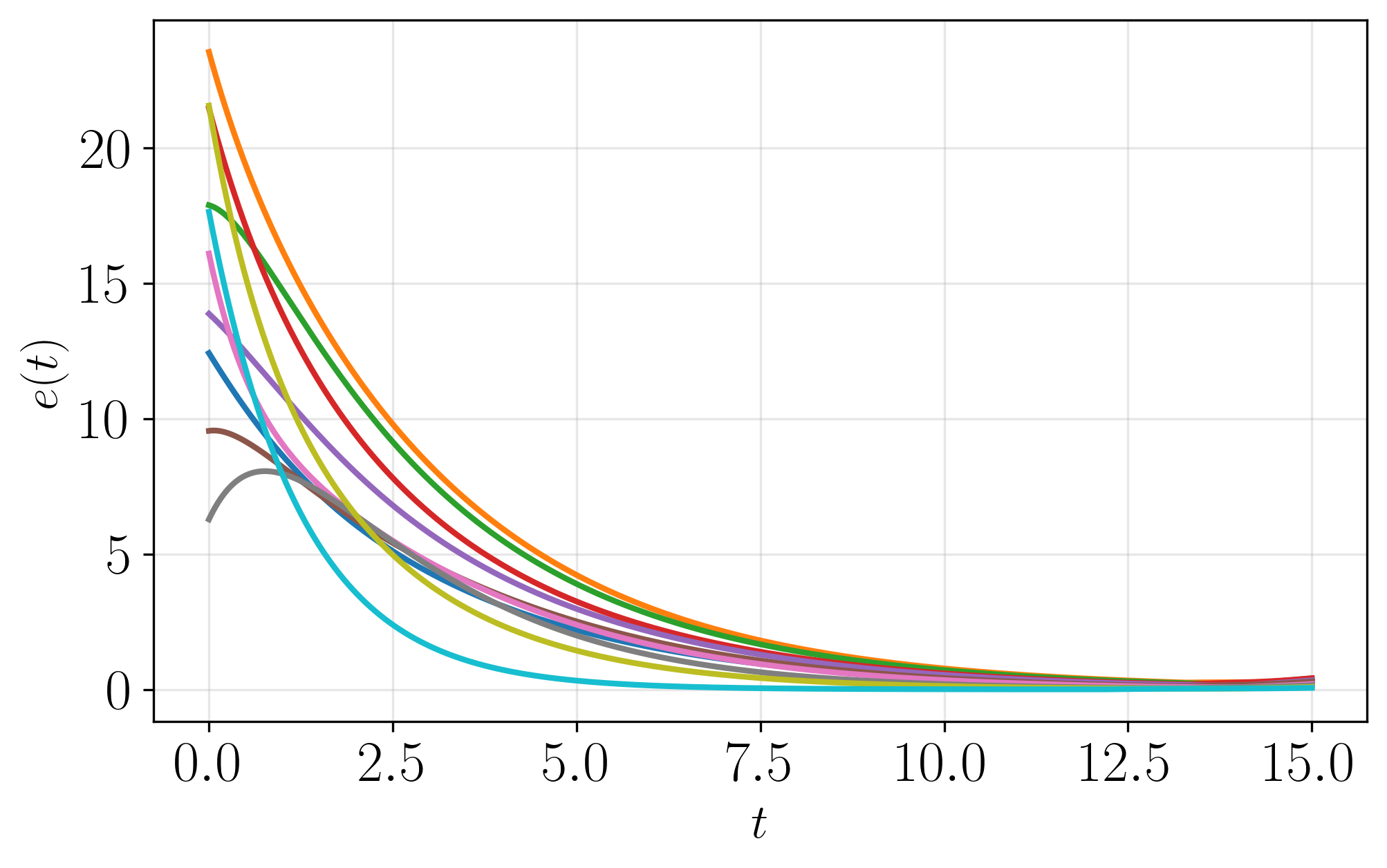}
        \caption{Finite Horizon: (Upper left) $q(t)$ over time. (Upper right) $[L(t)]_i$ over time. (Lower left) state trajectories (solid) vs estimator trajectories (dashed) over time. (Lower right) error system trajectories over time.}
        \label{fin-hor-1}
    \end{figure}
\begin{figure}[h]
        \centering
        \includegraphics[scale=0.35]{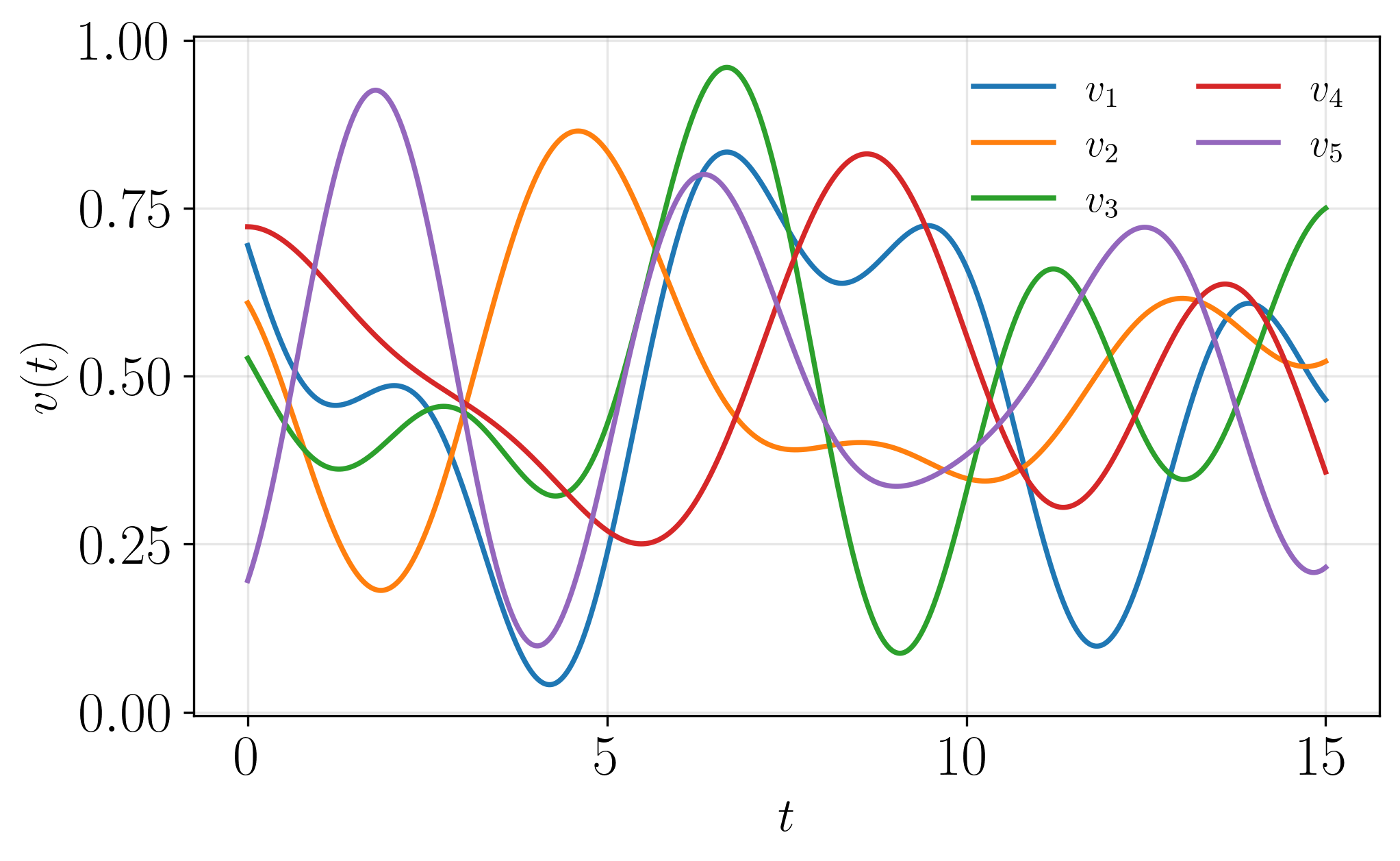}
        \includegraphics[scale=0.35]{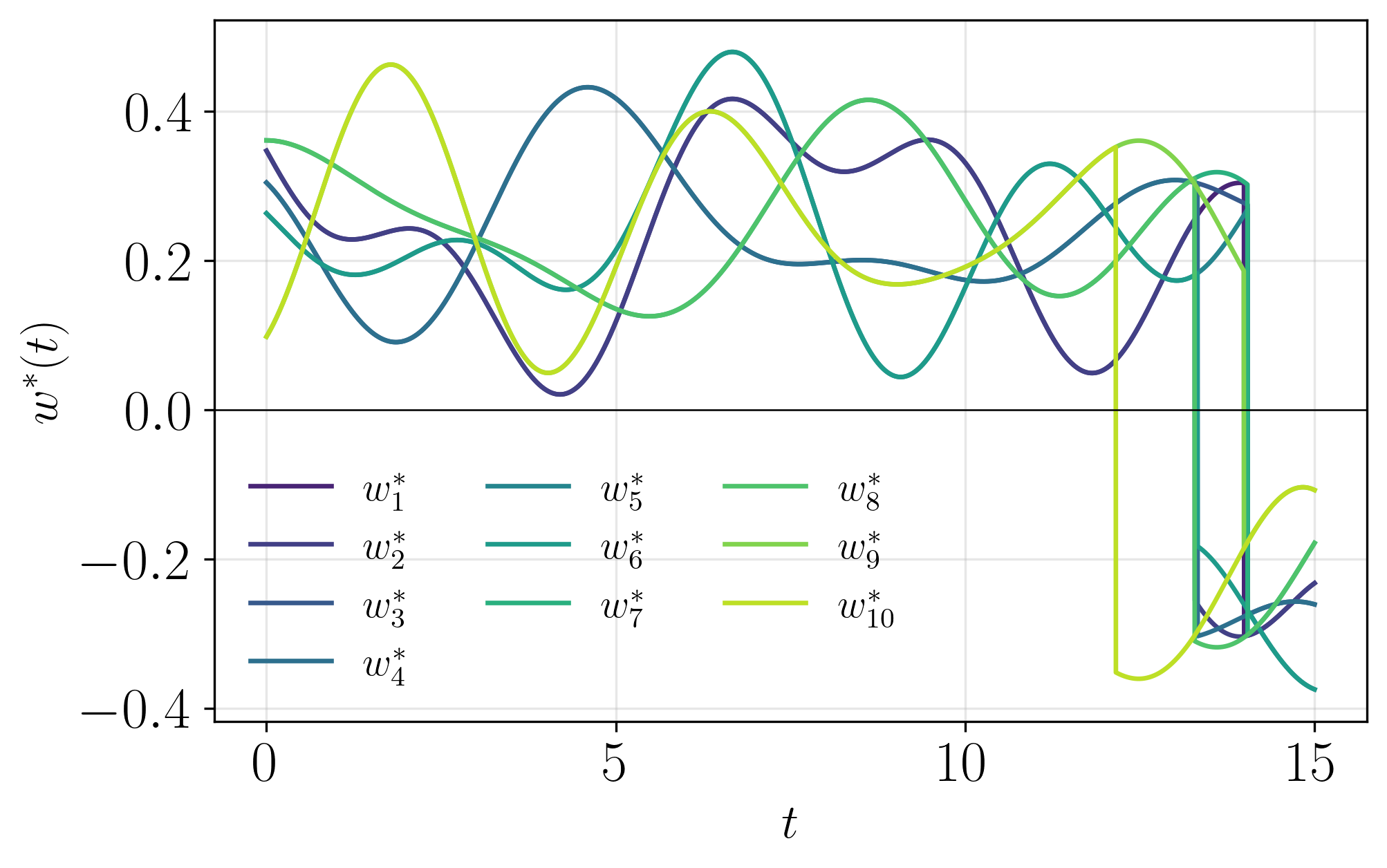}
        \caption{Measurement noise $v$ (left) and optimal disturbance $w^*$ (right) evolution over time.}
        \label{fin-hor-dist}
    \end{figure}
\subsubsection{Infinite Horizon} 
    \begin{figure}[h]
        \centering
        \includegraphics[scale=0.42]{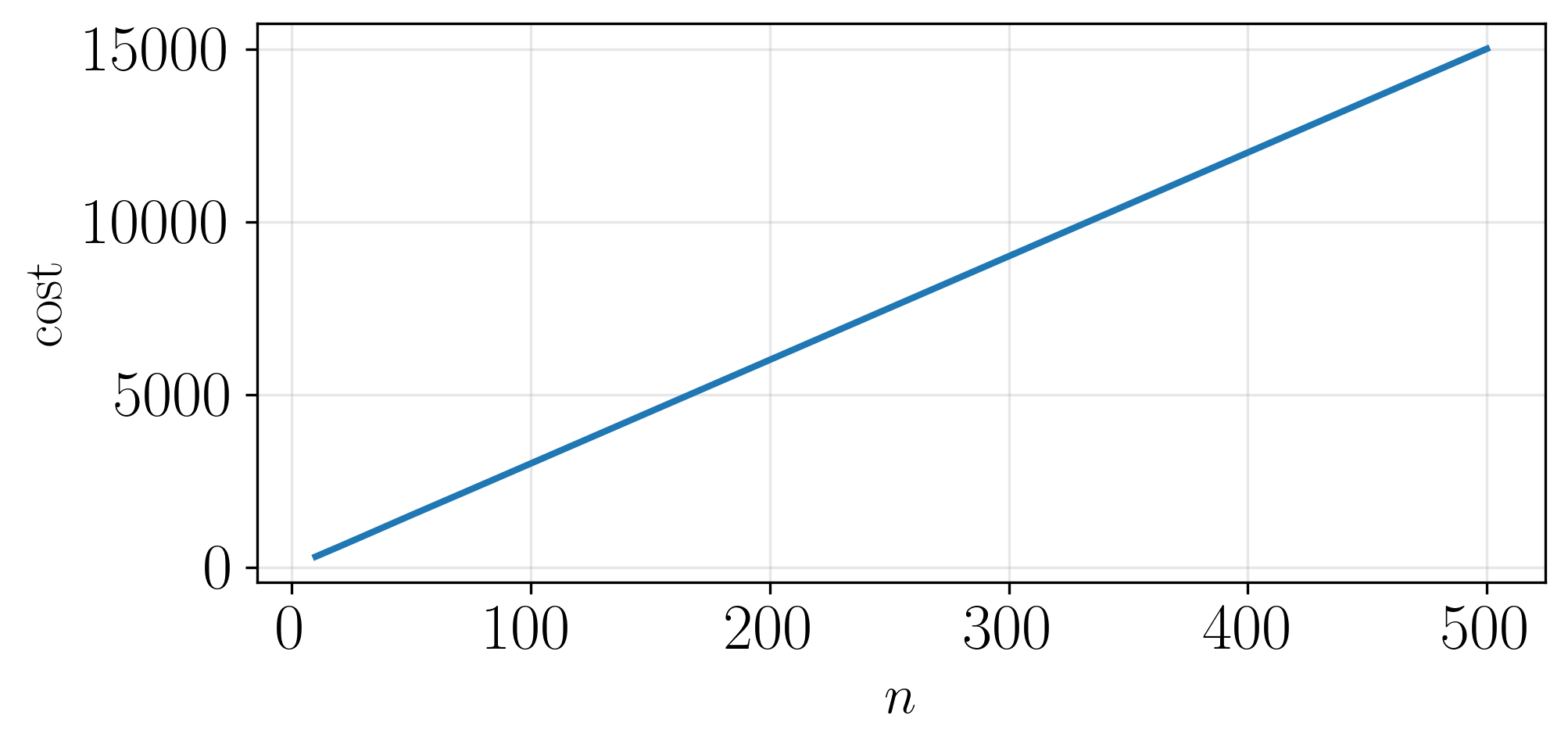}
        \caption{Cost of the stationary estimator as the network grows.}
        \label{inf-hor-2}
    \end{figure}
    \begin{figure}[h]
        \centering
        \includegraphics[scale=0.4]{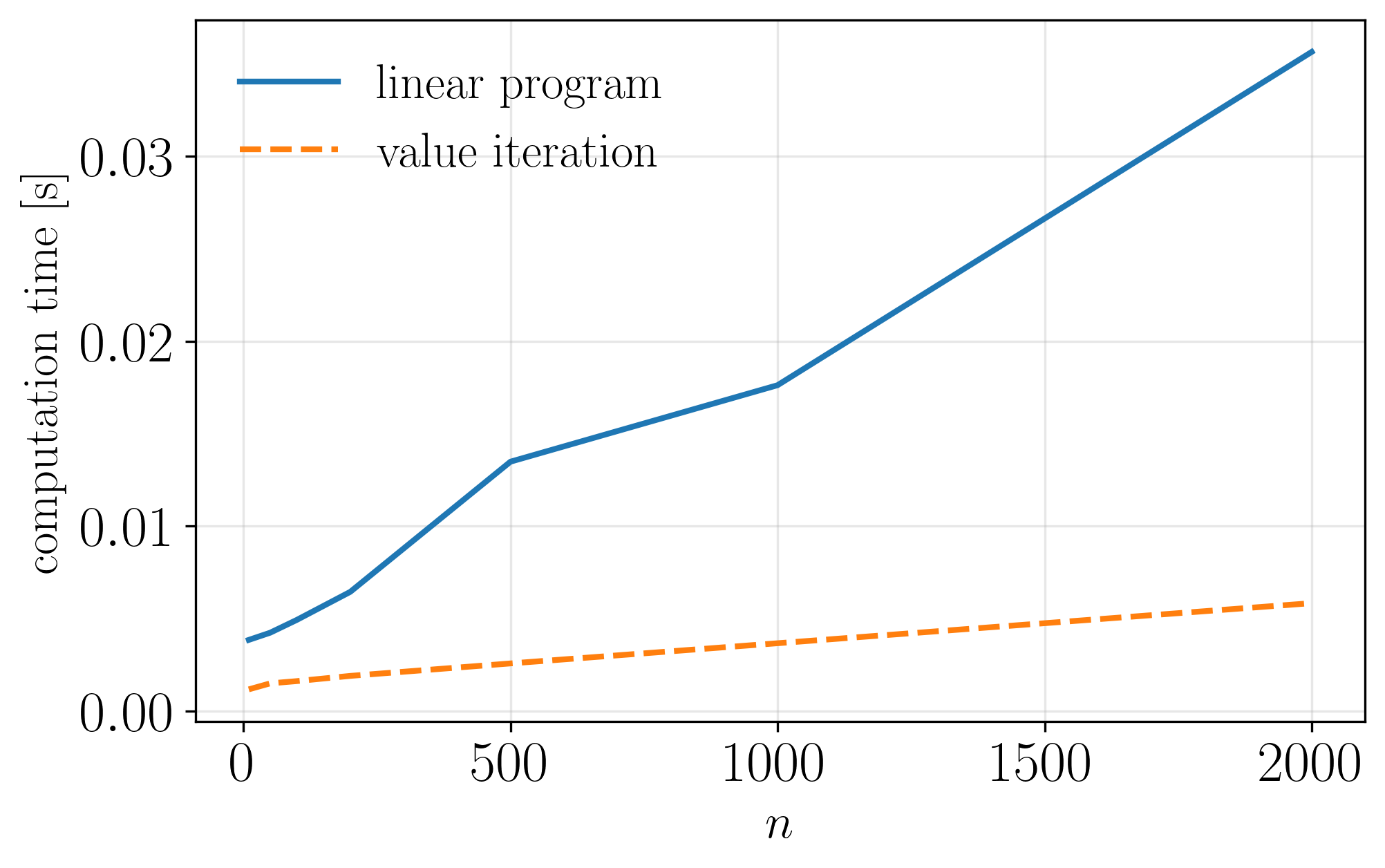}
        \caption{Time required to compute the stationary solution $q$ of~\eqref{ARE} with the linear program~\eqref{primal_LP} and the value iteration scheme~\eqref{p_seq}.}
        \label{inf-hor-3}
    \end{figure}
By Theorem~\ref{detec_apriori} with $h=20$, for which $A-HC+hI\geq 0$ and~\eqref{apriori_detect} holds, the pair $((\theta^{\top}+\psi^{\top}L)C, A-LC)$ is PZSD for every admissible gain. Hence, by Theorem~\ref{thm_q_LP1}, the minimal nonnegative solution $q$ of~\eqref{ARE} coincides with the optimizer of the linear program~\eqref{primal_LP}. For $n=10$, all entries of $\psi-q$ are negative, so the stationary gain is unique and equals $L=H$. It renders the error dynamics stable with $\alpha(A-LC)\approx-0.49$, so that under the disturbance $w=Lv$ the estimation error of the stationary estimator decays exponentially (Theorem~\ref{thm_detect_1}). Moreover, $q>0$, in agreement with Theorem~\ref{thm_detect_2}. Since $0<\psi_i<q_i$ for all $i$, the condition of Proposition~\ref{prop_static_optimal} does not hold, and the optimal gain equals $L=H$ except during the final $2.85$ time units, where its entries equal $-\eta$ (Remark~\ref{rem_optimal_inf}).

Figure~\ref{inf-hor-2} shows the cost of the stationary estimator over a window of length $T=40$. It grows approximately linearly with $n$, i.e. the cost per section does not deteriorate as the network grows. Moreover, it stays within $3\%$ of the minimum value (Remark~\ref{rem_optimal_inf}) for all sizes tested, e.g. $303$ against $296$ for $n=10$, so the stationary estimator is nearly optimal while being time-invariant. Figure~\ref{inf-hor-3} reports the time required to compute the stationary solution $q$ using the linear program~\eqref{primal_LP} and the value iteration~\eqref{p_seq}. Both scale approximately linearly with $n$, and the value iteration is consistently faster. Each of its iterations only requires a sparse matrix–vector product, whose cost grows linearly with $n$ for this network, while the number of iterations needed for convergence did not change with the network size in our experiments. The linear program, in contrast, is solved by an interior-point method, whose iterations involve the solution of linear systems with additional overhead. Both methods use sparse matrices (SciPy). The linear program is formulated with CVXPY and solved with the interior-point solver Clarabel. 
\vspace{-3mm}
\section{Conclusions}
This work introduces the $L_1$-minimum energy estimation problem for positive systems with state-dependent disturbances, and demonstrates that it admits explicit solutions through dynamic programming. Notably, the resulting optimal estimator is a Luenberger observer whose gain and optimal disturbance follow from a backward differential equation in the finite horizon, while in the infinite horizon an algebraic equation yields a stationary estimator with a static gain. We further proposed a condition under which the stationary estimator is optimal and when its associated gain stabilizes the error dynamics through positive zero-state detectability and $H$-detectability, which reflect that the gain must preserve positivity of the error in addition to being stabilizing. We further showed that the stationary solution can be computed by a linear program under appropriate detectability conditions and that the estimation error remains of finite $L_1$-norm under perturbations of finite $L_1$-norm. Since the number of unknowns of the stationary solution grows linearly with the state dimension and is associated with a sparse gain, the approach scales to large networks, as illustrated on an irrigation system.

Several directions remain open. An exciting direction of ongoing work is the joint design of a control input with the estimator, combining the present framework with the linear regulator for positive systems to obtain an output-feedback scheme with explicit solutions. A second direction concerns extending the framework to a general disturbance input matrix, and studying systematically how positive zero-state detectability and $H$-detectability relate to the detectability and controllability requirements of quadratic minimum-energy estimation. Another direction is to move beyond deterministic disturbances, either by introducing stochastic models that connect the linear-cost formulation to filtering under random noise, or by identifying other disturbance classes that retain the explicit form and linear-programming tractability exhibited here. It is also of relevance to identify applications for which the state-dependent disturbance constraints adopted here are a natural modeling choice. Finally, a worst-case formulation, in which the disturbance is chosen adversarially rather than to explain the observations, would yield a robust estimator under a linear performance measure for positive systems.
\vspace{-3mm}
\section{Acknowledgments}
The author would like to thank Prof. Anders Rantzer and Dr. Mark Jeeninga for their constructive feedback and discussions, which led to a significant improvement in the clarity and exposition of the paper.

 \bibliographystyle{abbrv} 
 \bibliography{cas-refs}
\end{document}